\documentclass[12pt]{amsart}
\usepackage[utf8]{inputenc}
\usepackage{amssymb,amsmath,latexsym,lipsum}
\usepackage{enumitem}
\usepackage{tabulary}
\usepackage{booktabs}
\usepackage{upgreek}
\usepackage{charter}
\usepackage[title]{appendix}
\usepackage{longtable}

\usepackage{amsthm}
\usepackage{comment}
\usepackage{euscript}
\usepackage{color}
\usepackage[T1]{fontenc}
\usepackage[utf8]{inputenc}
\usepackage{enumitem}
\usepackage[singlespacing]{setspace}
\usepackage[colorinlistoftodos]{todonotes}
\usepackage{tikz}
\usetikzlibrary {positioning}
\definecolor {processblue}{cmyk}{0.96,0,0,0}
\usepackage{bookmark}
\usepackage{multirow,bigdelim}
\usepackage{stackengine}
\usepackage{mathrsfs}
\usepackage{mathtools}
\usepackage{tikz-cd}
\usepackage{pgfplots}
\usepackage{pgf}
\usetikzlibrary{arrows}

\numberwithin{equation}{section}

\newcommand\undermat[2]{
  \makebox[0pt][l]{$\smash{\underbrace{\phantom{%
    \begin{matrix}#2\end{matrix}}}_{\text{$#1$}}}$}#2}

\usepackage{geometry}

\newtheorem{theorem}{Theorem}[section]
\newtheorem{proposition}[theorem]{Proposition}
\newtheorem{lemma}[theorem]{Lemma}
\newtheorem{conjecture}[theorem]{Conjecture}
\newtheorem{corollary}[theorem]{Corollary}
\newtheorem{definition}[theorem]{Definition}

\theoremstyle{remark}
\newtheorem{example}[theorem]{Example}
\theoremstyle{remark}
\newtheorem{remark}[theorem]{Remark}

\DeclareMathOperator{\Proj}{Proj}

\DeclareMathOperator{\cohom}{H}
\newcommand{\PP}{\mathbb{P}}
\newcommand{\Z}{\mathbb{Z}}
\newcommand{\C}{\mathcal{C}}
\newcommand{\Om}{\Omega}
\newcommand{\om}{\omega}

\newcommand{\ka}{\kappa}
\newcommand{\g}{\gamma}
\newcommand{\al}{\alpha}
\newcommand{\be}{\beta}
\newcommand{\Span}{\text{span}}
\newcommand{\im}{\text{im}}
\newcommand{\rank}{\text{rank}}
\newcommand{\rk}{\text{rk}}
\newcommand{\ord}{\text{ord}}

\newcommand{\steven}[1]{\todo[color=green!60]{SG: #1}}

\makeatother

\title{$a$-numbers of cyclic degree $p^2$ covers of the projective line}

\author[Dang]{Huy Dang}
\address[Dang]{National Center for Theoretical Sciences, Mathematics Division, No. 1, Sec. 4, Roosevelt Rd., Taipei City 106, Taiwan Room 503, Cosmology Building, National Taiwan University}
\email{huydang1130@ncts.ntu.edu.tw}

\author[Groen]{Steven R. Groen}
\address[Groen]{Department of Mathematics, 
         Lehigh University, 17 Memorial Drive E, 
         Bethlehem, Pennsylvania 18018}
\email{stg323@lehigh.edu}

\date{September 2023}

\pgfplotsset{compat=1.18}

\begin{document}

\begin{abstract}
We investigate the $a$-numbers of $\mathbb{Z}/p^2\mathbb{Z}$-covers in characteristic $p>2$ and extend a technique originally introduced by Farnell and Pries for $\mathbb{Z}/p\mathbb{Z}$-covers. As an application of our approach, we demonstrate that the $a$-numbers of ``minimal'' $\mathbb{Z}/9\mathbb{Z}$-covers can be deduced from the associated branching datum.
\end{abstract}

\maketitle

\section{Introduction}

Consider an algebraically closed field $k$ of characteristic $p > 2$, and let $\pi: Y \rightarrow X$ be a Galois branched cover of smooth, projective, connected curves. A prevalent research direction is to investigate to what extent invariants of $Y$ are determined by the branching datum of the cover $\pi$. See Section~\ref{secASWcovers} for a precise definition of the branching datum. For instance, if $Y$ is an elliptic curve, then the branching datum of $\pi$ does not determine whether $Y$ is ordinary or supersingular.

Of particular interest is the case when $p$ divides $ \lvert G \rvert$, i.e. when the cover is wildly ramified. Then the ramification of $\pi$ is more complicated, due to the additional structure of the filtration of higher ramification groups. In this paper we restrict to the case when $G$ is a cyclic $p$-group, and we write $G = \Z/p^n\Z$.

Some well-known invariants for curves in characteristic $p$ include the genus, $p$-rank, and $a$-number. Of particular interest to us is the $a$-number, which is defined by 
$$a_Y := \text{dim} \left( \text{ker}\left( \C: \cohom^0(Y,\Om_Y^1) \to \cohom^0(Y,\Om_Y^1) \right) \right),$$ 
where $\C$ is the $p^{-1}$-linear \emph{Cartier operator}. The size of $a_Y$ measures the failure of a curve to be \emph{ordinary}. The genus of $Y$ is determined by the branching datum of $\pi$ via the Riemann-Hurwitz formula \cite[IV. Corollary 2.4]{Hartshorne}. Similarly, since $G$ is a $p$-group, the Deuring-Shafarevich formula \cite[Theorem 4.1]{SubraoDS} dictates what the $p$-rank of $Y$ is in terms of the branching datum of $\pi$. In contrast, the $a$-number of $Y$ is not determined by the branching datum of $\pi$. Understanding the possibilities for the $a$-number of a curve has implications in the Schottky problem, as it helps illuminate which Abelian varieties are isomorphic to the Jacobian of a curve.

In recent years, there has been a growing interest in studying which $a$-numbers can occur for certain classes of curves. See, for instance, \cite{Voloch1988}, \cite{Johnston}, \cite{Elkin07hyperellipticcurves}, \cite{ELKIN20111}, \cite{Prieshyp2}, \cite{MR3090462}, \cite{Friedlander_2013}, \cite{Frei}, \cite{MPWSuzuki}, \cite{MR4113776}, \cite{booher_cais_2023}, \cite{2024arXiv240713966B} and \cite{booher2024higheranumbersmathbfzptowerscounting}. 

There has been particular interest in cyclic covers of $p$-power degree in characteristic $p$. The authors of \cite{MR3090462} focused on the special case where $X=\mathbb{P}^1$ and $G=\mathbb{Z}/p\mathbb{Z}$. They derived a formula for the $a$-number of $Y$ when the ramification jumps divide $p-1$. In particular, in this case the $a$-number is determined by the branching datum of $\pi$. Booher and Cais expanded the study in \cite{MR4113776} to a general base curve $X$ and established bounds for $a_Y$ based on $a_X$ and the branching datum of $\pi$. In \cite{booher_cais_2023}, the authors consider the (higher) $a$-numbers in towers of $\Z/p^n\Z$-covers of a base curve $X$, which in key cases is the projective line. In particular, they formulate far-reaching conjectures for the behavior of the $a$-numbers as a function of $n$ and provide theoretical and experimental evidence for these conjectures.

In recent works \cite{2024arXiv240713966B} and \cite{booher2024higheranumbersmathbfzptowerscounting}, parallel research on the $a$-numbers of $\mathbb{Z}/p^n\Z$-covers is conducted using a $p$-adic analytic approach, which differs from the method employed in this paper. Their method yields strong general results, but is limited to covers branched at a single point.

In this paper, we investigate covers of $\PP^1$ with Galois group $G=\Z/p^2\Z$ whose conductors are ``minimal'' (see Definition \ref{defnminimalcyclic}). This is analogous to the condition that the conductors divide $p-1$, as assumed in \cite{MR3090462}. Conjecture~\ref{conjnumberminimal} predicts that the $a$-number of a minimal $\Z/p^n\Z$-cover of $\mathbb{P}^1$ is determined by its branching datum. Extending the approach in \cite{MR3090462}, we assign to each basis differential $\omega$ of $\cohom^0(Y, \Om_Y^1)$ a \emph{key term} $\kappa(\C(\om))$, which is non-zero in $\C(\om)$. Using this, we prove a conditional lower bound for the rank of the Cartier operator, which is equivalent to giving an upper bound on the $a$-number. In this theorem $c_\omega$ is an element of $k$ associated to $\omega$ via the formula in Lemma~\ref{lemktcom}.

\begin{theorem} 
(Theorem~\ref{thmrkCgeqK}) Let $K$ be the set of key terms and let $H$ be the set of basis differentials that have a key term. Assume that for each $\om' \in K$, there is an $\om\in H$ such that $\ka(\C(\om))=\om'$ and $c_\om \neq 0$. Then we have $\rk(\C) \geq \#K$.
\end{theorem}

If we furthermore assume $p=3$, we prove this upper bound is sharp, resulting in the following theorem.

\begin{theorem} \label{thmp=3intro}
(Theorem~\ref{thmp=3}) The $a$-number of a minimal $\mathbb{Z}/9\mathbb{Z}$-cover in characteristic $3$ is determined by its branching datum.
\end{theorem}
See Theorem \ref{thmp=3} for a precise formula of $a_{Y_2}$ in terms of the branching datum.

\subsection{Outline}
The rest of the paper is organized as follows. In Section \ref{secbackground}, we provide background information on Artin-Schreier-Witt theory and the ramification of cyclic covers of the projective line in characteristic $p$. Additionally, we review the Cartier-Manin matrix and its relationship with the $a$-number of a curve. We also discuss the motivation behind our question and make comparisons with existing results in this section.

Section \ref{Sec:Keyterms} introduces the concept of key terms for $\mathbb{Z}/p^2\mathbb{Z}$-covers. The key terms represent leading entries of the Cartier-Manin matrix, generalizing the concept introduced in \cite{MR3090462} for $\mathbb{Z}/p\mathbb{Z}$-covers. It is shown in Theorem~\ref{thmrkCgeqK} that, under certain conditions, the number of key terms is a lower bound or the rank of the Cartier-Manin matrix, as differentials with a key term contribute to the rank. This yields a lower bound on the $a$-number.

Finally, in Section \ref{sec:p=3}, we demonstrate that the provided upper bound is sharp when $p=3$, by showing the remaining differentials do not contribute to the rank of the Cartier operator. This results in Theorem~\ref{thmp=3intro}.

\subsection{Comparison to Farnell-Pries}
As mentioned, in \cite{MR3090462} a formula is proved for the $a$-number of $\mathbb{Z}/p\mathbb{Z}$-covers $X \to \mathbb{P}^1$, assuming all the conductors divide $p-1$. In order to do so, a subset $H$ of the basis $W$ of the space of regular differentials is constructed. To differentials $\omega \in H$ a \emph{key term} $\kappa(\C(\om))$ is assigned. Moreover, Farnell and Pries define an order $\prec$ on $W$. They prove the following properties:
\begin{enumerate}
    \item The map $\kappa$ is injective. \label{item:inj}
    \item The coefficient of $\kappa(\C(\om))$ is non-zero in $\C(\om)$. \label{item:non-0}
    \item If $\om' \prec \om$, then the coefficient of $\kappa(\C(\om))$ is zero in $\C(\om')$. \label{item:0}
    \item $\C(\cohom^0(X,\Omega_X^1)) = \text{span}_k \{ \C(H) \}$. \label{item:surj}
\end{enumerate}
Together, these imply that the rank of $\C$ equals the cardinality of $H$.

The goal of the current paper is to extend this technique to $\mathbb{Z}/p^2\mathbb{Z}$-covers $Y_2 \to Y_1 \to \mathbb{P}^1$. The appropriately modified key terms are introduced in Definition~\ref{defkeyterm}. The analysis at a branch point $P$ works differently depending on whether $Y_1 \to \mathbb{P}^1$ is branched at $P$. In this context, \eqref{item:inj} no longer holds; it is possible for two differentials to have the same key term (see Remark~\ref{rmkincomparableiffsamekeyterm}). This is not a major issue, as one counts the cardinality of $K=\kappa(H)$ rather than $H$. As for \eqref{item:non-0}, a formula for the coefficient of $\kappa(\C(\om))$ is given in Lemma~\ref{lemktcom}. Although we suspect this is never zero, we are only able to prove this in specific cases, which is recorded in Lemma~\ref{lemcomnon-zero}. The analogue of \eqref{item:0} is proved in Lemma~\ref{lemktzero}. The differentials that have the same key term are incomparable with respect to the order $\prec$. These together suffice for the conditional upper bound on $a_{Y_2}$ that is proved in Theorem~\ref{thmrkCgeqK}. The condition of this upper bound is immediately satisfied if $p=3$. It is shown in Theorem~\ref{thmp=3} that the upper bound is sharp, meaning that \eqref{item:surj} is also satisfied when $p=3$.

\subsection{Acknowledgements}
The authors thank Jeremy Booher, Bryden Cais, Joe Kramer-Miller,  Rachel Pries, and Damiano Testa for helpful discussions. We would also like to thank the anonymous referees for useful suggestions.

\section{Background}
\label{secbackground}

\subsection{Artin-Schreier-Witt theory}
\label{secASWcovers}
According to the Artin-Schreier-Witt theory (as discussed in references such as \cite{MR1878556}), for any field $K$ with characteristic $p$, the following group isomorphism holds:
\begin{equation*}
    \cohom^1(G_K, \mathbb{Z}/p^n\mathbb{Z}) \cong W_n(K)/\wp (W_n(K)).
\end{equation*}
In the above equation, $G_K$ represents the absolute Galois group of $K$, $W_n(K)$ denotes the ring of length-$n$ Witt vectors over $K$, and $\wp(\underline{x}) = F(\underline{x}) - \underline{x}$ corresponds to the Artin-Schreier-Witt isogeny. We say two Witt vectors $\underline{f}$ and $\underline{g}$ in $W_n(K)$ belong to the same Artin-Schreier-Witt class if there exists a Witt vector $\underline{h} \in W_n(K)$ such that $\underline{f} = \underline{g} + \wp(\underline{h})$.

Therefore, any $\mathbb{Z}/p^n\mathbb{Z}$-cover $\phi_n: Y_n \xrightarrow{} \mathbb{P}^1$ can be represented by an Artin-Schreier-Witt equation of the form
        \begin{equation*}
            \wp(y_1, \ldots, y_n)=(f_1, \ldots, f_n),
        \end{equation*}
where $(f_1, \ldots, f_n) \in W_n(k(x))$ and $f_1 \not\in \wp(k(x))$. 
\begin{remark}
    For $1 \le i<n$, the unique $\mathbb{Z}/p^i$-sub-cover $\phi_i: Y_i \rightarrow \mathbb{P}^1$ of $\phi_n$ is given by the Artin-Schreier-Witt equation:
\begin{equation*}
    \wp(y_1, \ldots, y_i)=(f_1, \ldots, f_i).
\end{equation*}
\end{remark}

Alternatively, the cover $\phi_n$ can be expressed as a series of Artin-Schreier extensions, which is advantageous for computational purposes, as demonstrated later in this paper.
\begin{proposition}[{\cite{MR2577662}}]
\label{proptowerAS}
The $\mathbb{Z}/p^n\mathbb{Z}$-cover associated with $(f_1, \ldots, f_n) \in W_n(K)$ can be represented by the following system of equations: 
    \begin{equation}
    \label{eqntowerAS}
        \begin{split}
            y_1^p-y_1 & = f_1 \\
            y_2^p-y_2 & =g_1(y_1)+f_2 \\
            & \vdots \\
            y_n^p-y_n& = g_{n-1}(y_1, \ldots, y_{n-1}) + f_n,
        \end{split}
    \end{equation}
where $g_i \in \mathbb{F}_p[f_1, \ldots, f_i, y_1, \ldots, y_i]$ are explicitly defined. Specifically, $g_i$ can be given recursively in $\mathbb{Z}[f_1,\ldots, f_i,y_1, \ldots, y_i]$ as follows:
\begin{equation}
\label{eqntower}
g_{i}(f_1, \ldots, f_i, y_1, \ldots, y_i)=f_i+y_i+\sum_{d=1}^{i-1} \frac{1}{p^{i-d}}(f_d^{p^{i-d}}+y_d^{p^{i-d}}-g_{d-1}^{p^{i-d}}).
\end{equation}
\end{proposition}

In the case where $n=1$, we can eliminate all terms of $f_1$ with degrees divisible by $p$ by adding rational functions of the form $h^p-h$. This process can be generalized to the case $n>1$ using an induction technique similar to the proof presented in \cite[\S 26, Theorem 5]{MR2371763}. For a more comprehensive proof, please refer to \cite[Lemma A.2.3]{MR3714509}. To summarize this result, we state the following proposition. 

\begin{proposition}
\label{propreducedrep}
    Every $\underline{f} \in W_n(k(x))$ belongs to the same Artin-Schreier-Witt class as a vector $\underline{h}=(h_1, \ldots, h_n)$, where none of the terms in the partial fraction decomposition of $h_i$ has degree divisible by $p$.
\end{proposition}

We refer to such a vector $\underline{h}$ as \textit{reduced}. For the remainder of our discussion, we assume that $\underline{f}$ is reduced.

Let $B_n= {P_1, \ldots, P_r}$ denote the set of poles of the $f_i$'s, which is also the branch locus of $\phi_n$. At each ramified point $Q_j$ over $P_j$, $\phi_n$ induces an exponent-$p^n$ cyclic extension of the complete local ring $\hat{\mathcal{O}}_{Y_n,Q_j}/\hat{\mathcal{O}}_{\mathbb{P}^1,P_j}$. Consequently, we can derive the upper ramification filtration of $\phi_n$ at the branch point $P_j$ in a canonical manner.

Assuming that the inertia group of $Q_j$ is $\mathbb{Z}/p^m\mathbb{Z}$ (where $n \leq m$), for $i \leq n-m$, we define the \textit{$i$-th ramification break (jump)} of $\phi_n$ at $P_j$ as $-1$. For $i > n-m$, the $i$-th ramification break of $\phi_n$ at $P_j$ corresponds to the $(i-n+m)$-th break in $\hat{\mathcal{O}}_{Y_n,Q_j}/\hat{\mathcal{O}}_{\mathbb{P}^1,P_j}$. We denote the $i$-th upper ramification break of $\phi_n$ at $P_j$ as $u_{j,i}$. The \textit{$i$-th conductor of $\phi_n$ at $P_j$} is defined as $e_{j,i} := u_{j,i} + 1$. The following formula provides an explicit calculation of the ramification filtration of $\phi_n$ in terms of $\underline{f}$.

\begin{theorem}[{\cite[Theorem 1]{MR1935414}}]
\label{theoremcaljumpirred}
With the assumptions and the notations as above, we have
\begin{equation}
\label{eqnformulalowerjumpasw}
    u_{j,i}=\max\{ p^{i-l} \deg_{(x-P_j)^{-1}} (f_{l}) \mid l=1, \ldots, i\}, 
\end{equation}
for $i>n-m$.
\end{theorem}

\begin{proposition}
With the above settings, the genus of the curve $Y_i$ is
    \begin{equation}
    \label{eqngenera}
        g_{Y_i} = 1-p^i + \frac{\sum_{l=1}^i (\sum_{j=1}^r e_{j,l})(p^j-p^{j-1})}{2}
    \end{equation}
\end{proposition}

\begin{proof}
Applying the Grothendieck-Ogg-Shafarevich formula \cite[Exp. X formula 7.2]{MR0491704} to $\phi_i: Y_i \xrightarrow{} X$, we obtain the relation
\begin{equation*}
2g_{Y_i}-2=\deg(\phi_i) (2g_X-2) + \sum_{j=1}^r \deg\big(\mathscr{D}_{P_{j,i}}\big),
\end{equation*}
where is $\mathscr{D}_{P_{j,i}}$ the different of $\phi_i$ at $P_j$ \cite[IV]{MR554237}. Additionally, \cite[Fact 2.3]{MR3194816} asserts that
\begin{equation*}
\deg\big(\mathscr{D}_{P{j,i}}\big)=\sum_{l=1}^i e_{j,l}(p^l-p^{l-1}).
\end{equation*}
This and the fact that $g_X=0$ immediately implies the claim about the genus of $Y_i$.
\end{proof}
To record the branching datum of the cover, we use an $r \times n$ matrix of the form
\begin{equation}
\label{eqnmatrixbranchingdata}    
 \begin{bmatrix}
   e_{1,1} & e_{1,2} & \ldots & e_{1,n} \\
   e_{2,1} & e_{2,2} & \ldots & e_{n,n} \\
   \vdots & \vdots & \ddots & \vdots \\
   e_{r,1} & e_{r,2} & \ldots & e_{r,n} \\
   \end{bmatrix},
\end{equation}
which is referred to as the \textit{branching datum} of $\phi_n$. This notion is used to stratify the moduli space of cyclic covers of curves in \cite{2023arXiv230614711D}. The following is an immediate consequence of Theorem \ref{theoremcaljumpirred}.

\begin{proposition}
A matrix of the form (\ref{eqnmatrixbranchingdata}) is the branching datum of a $\mathbb{Z}/p^n\mathbb{Z}$-cover if and only if the following conditions hold:
\begin{enumerate}[label*=\arabic*.] 
\item \label{item1condbranchingdata} $e_{i,1} \not \equiv 1 \pmod{p}$,
\item \label{item2condbranchingdata} $e_{i,j} \ge pe_{i,j-1}-p+1$, and
\item \label{item3condbranchingdata} if $e_{i,j}>pe_{i,j-1}-p+1$, then $e_{i,j} \not\equiv 1 \pmod{p}$.
\end{enumerate}
\end{proposition}

\begin{remark}
The branching datum of $\phi_i$ is the $r \times i$ matrix that contains the first $i$ columns of (\ref{eqnmatrixbranchingdata}).
\end{remark}

\begin{example} 
Suppose $k$ is an algebraically closed field of characteristic $3$, and $\phi_2: Y_2 \xrightarrow{} \Proj k[x,z]$ is a $\mathbb{Z}/9\mathbb{Z}$-cover given by the following affine equation:
\begin{equation*}
    \wp(y_1, y_2)=\bigg( \frac{1}{x}+x, \frac{1}{x^5} -\frac{1}{x-1} \bigg)=:(f_1,f_2)=:\underline{f}.
\end{equation*}
Since none of the terms in $f_i$'s are a power of $3$, $\underline{f}$ is considered reduced. According to Theorem \ref{theoremcaljumpirred}, $\phi$ branches at $x=0$, $x=1$, and $x=\infty$, with the following branching datum:
\[
\begin{bmatrix}
2 & 6 \\
0 & 2 \\
2 & 4
\end{bmatrix}.
\]
The theorem also reveals that the ramification index of each ramified point above $x=1$ is $3$, while the unique point above $x=0$ or $x=\infty$ has a ramification index of $9$.

Finally, applying Proposition \ref{proptowerAS} allows us to write $\phi_2$ as a system of Artin-Schreier equations as follows:
\begin{equation*}
    \begin{split}
        y_1^3-y_1&= \frac{1}{x}+x=f_1, \\
        y_2^3-y_2&=g_1(y_1)+ \frac{1}{x^5}-\frac{1}{x-1},
    \end{split}
\end{equation*} 
where $g_1(y_1)= \frac{1}{3}(f_1^3+y_1^3-(f_1+y_1)^3)) =-y_1^7+y_1^5$.
\end{example}

\subsection{A basis for the space of regular differentials} 
We present a restatement of \cite[Lemma 5]{Madden78} using our established conventions. Consider a $\mathbb{Z}/p^n\mathbb{Z}$-Galois cover $Y_n \xrightarrow{\phi_n} \mathbb{P}^1$ whose branching datum is given by Equation \ref{eqnmatrixbranchingdata}. Let $B_n=\{P_1, \ldots, P_r\} \subset \mathbb{P}^1$ denote the branch locus of $\phi_n$, where $P_i$ corresponds to the $i$-th row of the branching datum. Without loss of generality, assume that $P_1$ is the point at infinity and $e_{1,1}>0$, so that $P_1$ has index $p^n$. We define $x_i:= \frac{1}{x-P_i}$ for $1<i \leq r$.

\begin{definition}
We define $N\in \mathbb{M}_{r\times n}(\mathbb{Z})$ as follows:
\begin{equation*}
N(i,j)=\begin{cases}
e_{i,j}-1 & \text{if } e_{i,j}>0 \\
0 & \text{if } e_{i,j}=0
\end{cases}
\end{equation*} 
For each $i=1, 2, \ldots, r$ and $j=1, \ldots, n$ such that $N(i,j)>0$, let $$E^j_i:=j-\min\{l \mid e_{i,l} \neq 0 \}+1.$$
\end{definition}

\begin{remark}
The positive integer $p^{E^j_i}$ represents the ramification index of each ramified point of $Y_j \xrightarrow{\phi_j} \mathbb{P}^1$ above $P_i$.
\end{remark}

\begin{definition}
We define $\lambda(i,j)$ recursively as follows:
\begin{equation*}
\lambda(i,j)=p^{E^j_i}N(i,j)-(p-1)\sum_{l=1}^{j-1}p^{l+E^j_i-j-1}N(i,l).
\end{equation*}
\end{definition}

Now we can rephrase \cite[Lemma 5]{Madden78} using our notations.

\begin{proposition} \label{propMadden}
    With the above settings, the $k$-vector space $\cohom^0(Y_n, \Omega^1_{Y_n/k})$ has a basis given by $\bigcup_{i=1}^r W_i$, where
    \begin{equation}
    \begin{aligned} 
    W_{1}:= \begin{cases} y_n^{a_n} \ldots y_1^{a_1} x^v dx \; | \; 0 \leq a_i <p, \\  
       0 \le p^n v \le (p-1)d_1-A_1-p^n-1 \end{cases}  \Bigg\},
    \end{aligned}
    \end{equation}
    and, for $i>1$,
    \begin{equation}
    \begin{aligned}
    W_i:= \begin{cases} y_n^{a_n} \ldots y_1^{a_1} x_i^v dx \; | \; 0 \leq a_i <p, \\  
       0 < p^{e_i} v \leq  (p-1)d_i -A_i+p^{e_i}-1 \end{cases}  \Bigg\},
\end{aligned}
\end{equation}
where 
\begin{equation*}
    d_i:=\sum_{\nu =1}^n \lambda(i,\nu), \text{ and } A_i:=\sum_{\nu =1}^n a_{\nu}\lambda(i,\nu).
\end{equation*}
\end{proposition}

\subsection{\texorpdfstring{$a$}{a}-number of a curve}
\label{secanumber}
Suppose $C$ is a curve of genus $g$ over field $k$. We let $\sigma$ denote the $p$-power Frobenius automorphism of $k$, which induces a pull-back through the absolute Frobenius $F^{\ast}_X: \cohom^1(X, \mathcal{O}_X) \xrightarrow{} \cohom^1(X, \mathcal{O}_X)$. The Cartier operator is a $\sigma^{-1}$-linear map  
\begin{equation*}
    \mathcal{C}:  \cohom^0 (C, \Omega^1_{C/k}) \rightarrow \cohom^0 (C, \Omega^1_{C/k}).
\end{equation*}
It is the dual of $F^{\ast}_X$ through Grothendieck–Serre duality. The $a$-number of a curve $C$, denoted as $a_C$, is defined as the dimension of the kernel of $\mathcal{C}$.

Assuming we have a basis $\beta=\{ \omega_1, \ldots, \omega_g \}$ of $\cohom^0(C, \Omega^1_{C/k})$, for each $\omega_j$ we can find coefficients $m_{j,i} \in k$ such that
\begin{equation*}
    \C (\omega_j) =\sum_{i=1}^g m_{j,i} \omega_i.
\end{equation*}
The resulting $(g \times g)$-matrix $M=(m_{j,i})$ is known as a \textit{Cartier-Manin matrix}. Consequently, the $a$-number can be computed as $a_C=g_C - \rank (M)$.

\begin{example}
\begin{enumerate}
    \item An ordinary elliptic curve has $p$-rank $1$ and $a$-number $0$.
    \item A supersingular elliptic curve has $p$-rank $0$ and $a$-number $1$.
\end{enumerate}
\end{example}

Generalizing the case of elliptic curves, a curve is said to be ordinary if its $a$-number is zero. Therefore, the $a$-number of a curve can be seen as a measure of how far the curve is from being ordinary.

Unlike the genus and the $p$-rank, the $a$-number of a curve cover is not solely determined by the base curve and the branching datum. An example of a pair of $\Z/p\Z$-covers with identical branching datum, but different $a$-number, is given in \cite[Example 4.6]{MR4113776}.

\subsection{\texorpdfstring{$\mathbb{Z}/p^n\mathbb{Z}$}{Z/p^n}-covers with minimial jumps}

\begin{definition}
\label{defnminimalcyclic}
A $\mathbb{Z}/p^n\mathbb{Z}$-cover with branching datum (\ref{eqnmatrixbranchingdata}) is said to be \emph{minimal} if it satisfies the following conditions:
\begin{enumerate}
\item If $j= \min \{k \mid e_{i, k} \neq 0 \}$, then $(e_{i,j}-1) \mid (p-1)$, and
\item $e_{i,l}-1=p^{l-j}(e_{i,j}-1)$ for $j<l \le n$. 
\end{enumerate}
\end{definition}

The $a$-number of a minimal $\mathbb{Z}/p\mathbb{Z}$-cover is known \cite[Theorem 1.1]{MR3090462}. Our computed examples suggest the following conjecture.

\begin{conjecture}
\label{conjnumberminimal}
Let $\phi_n$ be a minimal branched cover of $\mathbb{P}^1$ with Galois group $\mathbb{Z}/p^n\mathbb{Z}$. Then the $a$-number of $\phi_n$ is uniquely determined by its branching datum.
\end{conjecture}

In this paper, we focus on the case $n=2$. That allows us to only consider covers with branching data of the form
\begin{equation*}
\begin{bmatrix}
2&  \ldots  & a_i & \ldots  &  0 & \ldots & 0 & \ldots  \\ 
p+1&  \ldots  & pa_i-p+1 & \ldots  & 2& \ldots & a_i & \ldots  \\ 
\end{bmatrix}^{\intercal},
\end{equation*}
where $(a_i-1) \mid (p-1)$. In Section~\ref{Sec:Keyterms}, we generalize the methods in \cite{MR3090462}, based on key terms, for $\mathbb{Z}/p^2\mathbb{Z}$-covers. As an application of this new method, we prove that the answer is affirmative for the case when $p=3$ and $n=2$.

\begin{theorem}
\label{theoremminimal3}
Conjecture~\ref{conjnumberminimal} holds when $p=3$ and $n=2$ (see Theorem~\ref{thmp=3}).
\end{theorem}

\begin{remark}
Note that the Conjecture~\ref{conjnumberminimal} certainly doesn't hold for covers that are not minimal, as is demonstrated in \cite[Example 7.2]{MR4113776}.
\end{remark}

\subsection{Comparison with the known results}

In this section, we compare our results and conjectures with those that are already known for Artin-Schreier covers. Consider a $\mathbb{Z}/p^n\mathbb{Z}$-Galois cover $\phi_n: Y_n \xrightarrow{} \ldots \xrightarrow{} Y_{1} \xrightarrow{} Y_0$ of curves. Since $\phi_{n/(n-1)}: Y_n \xrightarrow{} Y_{n-1}$ is an Artin-Schreier cover, we can establish bounds on the $a$-number of $Y_n$ inductively, building upon what is known for $\mathbb{Z}/p\mathbb{Z}$-covers.

Suppose $P \in Y_0$ is a branch point with ramification index $p^n$. Let $(m_1, \ldots, m_n)$ denote the upper ramification jumps of $\phi_n$ at $P$. By applying the Herbrand function and its inverse, one can compute the (unique) upper ramification jump of $\phi_{n/(n-1)}$ at the branch point $P_{n-1} \in Y_{n-1}$ that lies above $P$.

\begin{proposition}
With the above notation, the jump of the intermediate Artin–Schreier extension $Y_n \to Y_{n-1}$ at $P_{n-1}$ is 
\[
\widetilde{m}_n = p^{n-1}m_n - \sum_{i=1}^{n-1} (p-1)p^{i-1} m_i.
\]
\end{proposition}

\begin{proof}
Localizing the covers $Y_n \to Y_{n-1} \to Y_0$ at the point $P$ yields Galois extensions of complete discrete valuation fields $K_n/K_{n-1}/K_0$ with Galois group $G := \mathrm{Gal}(K_n/K_0) \cong \mathbb{Z}/p^n\mathbb{Z}$ and upper jumps $(m_1, \ldots, m_n)$. Let $H = \mathrm{Gal}(K_n/K_{n-1}) \cong \mathbb{Z}/p\mathbb{Z}$ be the unique order-$p$ subgroup of $G$.

By the compatibility of the lower numbering with subgroups (\cite[IV, §3, Prop.~14]{MR554237}), the lower numbering filtration on $H$ is given by $H_t = H \cap G_t$ for all $t \ge 0$. Since $G_t$ drops from order $p$ to $1$ at $t = \psi_G(m_n)$, where $\psi_G$ is the inverse Herbrand function, we find that this is precisely when $H_t$ becomes trivial. Hence, both the lower and upper jumps of $H$ are given by $\widetilde{m}_n = \psi_G(m_n)$.

One can check that $\psi_G(x)$ for $x \in [m_{n-1}, m_n]$ takes the form
\[
\psi_G(x) = p^{n-1}x - \sum_{i=1}^{n-1} (p-1)p^{i-1} m_i,
\]
and evaluating $\psi_G$ at $x = m_n$ gives the desired formula.
\end{proof}

Applying the above result to the case $n=2$ and $m_2=pm_1$ yields:

\begin{corollary}
In the situation of Theorem \ref{theoremminimal3}, the jump of $Y_2 \xrightarrow{} Y_1$ at $P_1$ is given by
\begin{equation*}
\widetilde{m} = pm_2-m_1(p-1) = m_1(p^2-p+1).
\end{equation*}
\end{corollary}

\begin{remark}
    The jump $\widetilde{m}$ does not divide $p-1$, hence is not covered by \cite{MR3090462}. 
\end{remark}

Let's now recall the following results for the $a$-number of a $\mathbb{Z}/p\mathbb{Z}$-cover using the language of branching data.

\begin{theorem}[{\cite[Theorem 1.1]{MR4113776}}] \label{thmBC}
Suppose $\phi: Y \rightarrow X$ is a $\mathbb{Z}/p\mathbb{Z}$-Galois cover with branching datum
    \begin{equation*}
        \begin{bmatrix}
            d_{1,1} +1 & d_{2,1}+1 & \ldots & d_{r,1}+1 \\
    \end{bmatrix}^\top.
    \end{equation*}
Then the $a$-number of $Y$ is bounded as follows
\begin{equation*}
    \sum_{l=1}^r \sum_{i=j}^{p-1} \bigg(  \bigg\lfloor \frac{id_{l,1}}{p} \bigg\rfloor - \bigg\lfloor \frac{id_{l,1}}{p} -\bigg(1 -\frac{1}{p} \bigg) \frac{jd_{l,1}}{p} \bigg\rfloor  \bigg) \le a_Y \le pa_X+\sum_{l=1}^r \sum_{i=1}^{p-1} \bigg(  \bigg\lfloor \frac{id_{l,1}}{p} \bigg\rfloor - (p-i) \bigg\lfloor \frac{id_{l,1}}{p^2} \bigg\rfloor  \bigg).
\end{equation*}
\end{theorem}

\begin{proposition}[{\cite{MR3090462}}]
    Suppose a $\mathbb{Z}/p\mathbb{Z}$-cover $\phi: Y \rightarrow \mathbb{P}^1$ has branching datum
    \begin{equation*}
        \begin{bmatrix}
            d_{1,1}+1 & d_{2,1}+1 & \ldots & d_{r,1}+1  \\
    \end{bmatrix}^\top.
    \end{equation*}
    Suppose moreover that $d_{i,1} \mid (p-1)$ for all $i=1, 2, \ldots, r$. Then the $a$-number of $Y$ is
    \begin{equation*}
        a_Y=\sum_{i=1}^r a_i \text{, where } a_i=\begin{cases} 
      \frac{(p-1)d_{i,1}}{4} & \text{if $d_{i,1}$ is even} \\
      \frac{(p-1)(d^2_{i,1}-1)}{4} & \text{if $d_{i,1}$ is odd.} \\
   \end{cases}
    \end{equation*}
\end{proposition}

We now can compute the lower bound of the $a$-number for the case $Y_2 \xrightarrow{} Y_0$ is minimal-one-point cover whose first jump equals $1$.

\begin{proposition} \label{proplowerbound}
Suppose $p=2q+1$. Then
    \begin{equation*}
        \sum_{i=q+1}^{p-1} \bigg(  \bigg\lfloor \frac{i (p^2-p+1) }{p} \bigg\rfloor - \bigg\lfloor \frac{i(p^2-p+1)}{p} -\bigg(1 -\frac{1}{p} \bigg) \frac{(q+1)(p^2-p+1)}{p} \bigg\rfloor  \bigg)
    \end{equation*}
is equal to $\frac{p(p-1)^2}{4}$.
\end{proposition}

\begin{proof}
    By a straightforward computation, we can indeed see that each summand is equal to $q(2q+1)$. Since there are $q$ such summands, the result follows immediately.
\end{proof}

Suppose $p = 2q + 1$ and $X$ is ordinary. Suppose $\phi_2: Y_2 \to X$ is a minimal $\mathbb{Z}/p^2\mathbb{Z}$-cover branched at one point $P$, with first jump $d_{1,1}$. Let $\tilde{d}_{1,1}$ be the ramification jump of the intermediate subextension $Y_2 \to Y_2 / (\mathbb{Z}/p)$, corresponding to the unique order-$p$ subgroup of the Galois group of $Y_2 \to X$, at the point above $P$. Combining Theorem~\ref{thmBC} and Proposition~\ref{proplowerbound}, we obtain the following result.

\begin{corollary}
    In the above setting, we have
    \begin{equation*}
    \sum_{i=q+1}^{p-1} \left( \left\lfloor \frac{i \tilde{d}_{1,1}}{p} \right\rfloor - \left\lfloor \frac{i \tilde{d}_{1,1}}{p} - \left( 1 - \frac{1}{p} \right) \frac{j \tilde{d}_{1,1}}{p} \right\rfloor \right) \le a_{Y_2} 
        \le p a_{Y_1} + \sum_{i=1}^{p-1} \left( \left\lfloor \frac{i \tilde{d}_{1,1}}{p} \right\rfloor - (p-i) \left\lfloor \frac{i \tilde{d}_{1,1}}{p^2} \right\rfloor \right).
    \end{equation*}
    In particular, when $d_{1,1} = 1$, hence $\tilde{d}_{1,1} = p^2 - p + 1$, and $X \cong \mathbb{P}^1$, with $a_{Y_1} = 0$, we have
    \begin{equation*}
        \frac{p(p-1)^2}{4} \le a_{Y_2}.
    \end{equation*}
\end{corollary}

Since we expect the $a$-number of a minimal $\Z/p^n\Z$-cover of $\PP^1$ to be determined by the branching datum, and this lower bound is attained in examples, we conjecture that this lower bound is always sharp for this class of curves. On the other hand, in the case $d_{1,1}>1$ the lower bound for the $a$-numbers of minimal $\Z/p^2\Z$-covers obtained by applying \ref{thmBC} twice does not seem to be sharp.

\section{Key terms} \label{Sec:Keyterms}

Let $\phi_2: Y_2 \to \mathbb{P}^1$ be a minimal $\Z/p^2\Z$ cover, with subcover $\phi_1: Y_1 \to \PP^1$. Using the identification $k(\PP^1) \cong k(x)$ and Proposition~\ref{proptowerAS}, we obtain the models
\begin{align}
    Y_1:y_1^p-y_1 &= f(x)  \label{eqmodelY1}\\
    Y_2:y_2^p-y_2 &= g_1(y_1)+h(x). \label{eqmodelY2}
\end{align}
For brevity we write $g(y_1):=g_1(y_1)$. Explicitly this polynomial is given by
\begin{equation} \label{eqdefg}
 g(y_1):= \sum_{i=1}^{p-1} (-1)^{i} \frac{(p-1)!}{i!(p-i)!} y_1^{p(p-i)+i}.
\end{equation}
Equivalently, the cover $Y_2 \to \mathbb{P}^1$ is represented by the Witt vector $(f,h) \in W_2(k(x))$, with $(f,h)$ reduced in the sense of Proposition~\ref{propreducedrep}.

In this section, we define for certain regular differentials $\om \in \cohom^0(Y_2, \Om_{Y_2}^1)$ a \emph{key term} $\ka(\C(\om))$, analogous to \cite[Definition 3.2]{MR3090462}. An analysis of the key terms leads to a lower bound for the rank of the Cartier-Manin matrix, or equivalently an upper bound for the $a$-number $a_{Y_2}$. This upper bound is given in Theorem~\ref{thmrkCgeqK}.

Denote by $B_2$ the branch locus of the cover $\phi_2: Y_2 \to \PP^1$. For $P \in B_2$, let $x_P \in k(x)$ be such that $\ord_P(x_P)=-1$. Then the space $\cohom^0(Y_2,\Om_{Y_2}^1)$ is spanned by differentials of the form $y_2^{a_2} y_1^{a_1} x_P^v dx$, as described in Proposition~\ref{propMadden}. We consider two cases, depending on whether the cover $\phi_1: Y_1 \to \PP^1$ is already branched at $P$. 

If $P$ is a branch point of $\phi_1: Y_1 \to \PP^1$, with ramification jump $d_P$, then $f$ has a pole at $P$ of order $d_P$. To simplify proofs, we pick $f$ to have no further poles. Since the cover $\phi_2: Y_2 \to \PP^1$ is minimal, it follows that $d_P | p-1$. This case is discussed in Section~\ref{subsec:ktpolef}.

If $\phi_1: Y_1 \to \PP^1$ is \'{e}tale at $P$, then $f$ is regular at $P$. By the assumption that $\phi_2: Y_2 \to \PP^1$ is branched at $P$, the function $h \in k(x)$ has a pole at $P$. The order of this pole, which we denote by $e_P$, equals the ramification jump at $P$ of the cover $Y_2 \to Y_1$. This case is treated in Section~\ref{subsec:ktpoleh}.

Finally, combining both types of key terms yields the conditional upper bound on the $a$-number proved in Theorem~\ref{thmrkCgeqK}. This conditional upper bound is strong enough to prove Conjecture~\ref{conjnumberminimal} in the case $p=3$ and $n=2$.

\subsection{Key terms at poles of \texorpdfstring{$f$}{}} \label{subsec:ktpolef}

Let $B_1 \subset \PP^1$ be the branch locus of $\phi_1: Y_1 \to \PP^1$ and let $P \in B_1$. After using an automorphism of $\PP^1$ if necessary, we may assume $\infty \in B_1$. Since the cover is minimal (see Definition~\ref{defnminimalcyclic}), the pole orders at the points $P\in B_1$, defined by $d_P=-\ord_P(f)$, all divide $p-1$. Furthermore, since the upper jumps are minimal by Definition~\ref{defnminimalcyclic}, one may assume the pole order of $h$ at $P$ is bounded: $\ord_P(h) >- pd_P$. We write $x_\infty = x$ and $x_P=(x-P)^{-1}$. Then partial fraction decomposition allows us to write

\begin{equation} \label{eqf=sumf_P}
    f(x)= \sum_{P \in B} f_P(x_P),
\end{equation}
where $f_P \in k[x_P]$ is a polynomial of degree $d_P$. Let $u_P \in k^\times$ denote the leading coefficient of $f_P$. Without loss of generality, we assume $u_\infty=1$.

Using Proposition~\ref{propMadden}, we obtain the following formulas for $W_P$.
\begin{equation} 
\begin{aligned} \label{eqdefWinf}
    W_\infty:= \begin{cases} y_2^{a_2} y_1^{a_1} x^v dx \; | \; 0 \leq a_1,a_2 <p, \\  
       0 \leq p^2 v \leq  pd_P(p-1-a_1) + d_P(p^2-p+1)(p-1-a_2)-p^2-1 \end{cases}  \Bigg\}.
\end{aligned}
\end{equation}

\begin{equation}
\begin{aligned} \label{eqdefWP}
    W_P:= \begin{cases} y_2^{a_2} y_1^{a_1} x_P^v dx \; | \; 0 \leq a_1,a_2 <p, \\  
       0 < p^2 v \leq  pd_P(p-1-a_1) + d_P(p^2-p+1)(p-1-a_2)+p^2-1 \end{cases}  \Bigg\}.
\end{aligned}
\end{equation}

Following \cite{MR3090462}, define $\epsilon_P=-1$ if $P=\infty$ and $\epsilon_P=1$ otherwise. Then, as the Cartier operator on $k(x)dx$ is well understood, we obtain
\begin{equation*} \label{eqCxe_P}
\C(x_P^{ap+\epsilon_P} dx) = x_P^{a+\epsilon_P} dx.
\end{equation*}

The approach we take in this paper relies essentially on reducing the action of the Cartier operator on $\cohom^0(Y_2,\Omega_{Y_2}^1)$ to the Cartier operator on $k(x)dx$, through the models of $Y_2$ and $Y_1$.

Using $a_1,a_2 \geq 0$, we get
\begin{equation} \label{eqv<pd1}
v \leq \frac{1}{p^2}(pd_P(p-1)+d_P(p^2-p+1)(p-1)+\epsilon_Pp^2-1) < pd_P +\epsilon_P.
\end{equation}
The following lemma allows us to reduce computations involving the Cartier operator on $Y_2$ to computations involving the Cartier operator on $Y_1$.

\begin{lemma} \label{lemC(om)}
Let $\om:= y_2^{a_2} y_1^{a_1} x_P^v dx$ be an element of $\bigcup_{P \in B_1}W_P$. Then we have
\begin{equation} \label{eqCW}
\C(\om)= \sum_{0 \leq j+l \leq a_2} (-1)^{a_2-j} \binom{a_2}{j \, l} y_2^j \C(g(y_1)^l h(x)^{a_2-j-l} y_1^{a_1} x_P^v dx).
\end{equation}
\end{lemma}

\begin{proof}
Substituting $y_2=y_2^p-g(y_1)-h(x)$, from Equation \eqref{eqmodelY2} and expanding gives
\begin{align*}
\om&=(y_2^p-g(y_1)-h(x))^{a_2} y_1^{a_1} x_P^v dx \\
&= \sum_{0 \leq j+l \leq a_2} (-1)^{a_2-j} \binom{a_2}{j \, l} y_2^{pj} g(y_1)^{l} h(x)^{a_2-j-l} y_1^{a_1} x_P^v dx.
\end{align*}
Then, using the $p^{-1}$-linearity of the Cartier operator, we can pull out the $y_2$ to obtain
\begin{align*}
\C(\om) &= \C\left(\sum_{0 \leq j+l \leq a_2} (-1)^{a_2-j} \binom{a_2}{j \; l} y_2^{pj} g(y_1)^{l} h(x)^{a_2-j-l} y_1^{a_1} x_P^v dx \right) \\ 
&= \sum_{0 \leq j+l \leq a_2} (-1)^{a_2-j} \binom{a_2}{j \, l} y_2^{j} \C(g(y_1)^{l} h(x)^{a_2-j-l} y_1^{a_1} x_P^v dx),
\end{align*}
as desired.
\end{proof}

In order to compute the $a$-number of $Y_2$, we bound the rank of $\C$ from below. We do so using a differential we'll call the \emph{key term}. Recall that we have assumed that $d_P | p-1$, allowing us to define $\gamma_P := \frac{p-1}{d_P} \in \mathbb{Z}$. 

\begin{definition} \label{defkeyterm}
Suppose $P \in B_1$. Let 
\begin{align*}
\al_P(v)&:=\left \lfloor \frac{\g_P(v-\epsilon_P)}{p} \right \rfloor\\
\be_P(v)&:=\g_P(v-\epsilon_P)-p\al_P(v). 
\end{align*}
We define 
\begin{equation*}
H_P := \{ y_2^{a_2} y_1^{a_1} x_P^v dx \in W_P \; | \; \be_P(v) \leq a_1+pa_2 \}
\end{equation*}
For $\om:= y_2^{a_2} y_1^{a_1} x_P^v dx \in H_P$, we define the associated \emph{key term} $\ka(\C(\om))$ as follows:
\begin{equation} \label{eqdefkeyterm}
\ka(\C(\om)):= \begin{cases} 
y_2^{a_2}y_1^{a_1-\be_P(v)}x_P^{v-d_P\al_P(v)} dx  &\hbox{\textup{if} $a_1 \geq \be_P(v)$} \\
y_2^{a_2-1}y_1^{a_1-\be_P(v)+p}x_P^{v-d_P\al_P(v)} dx &\hbox{\textup{if} $a_1 < \be_P(v)$}.
\end{cases}
\end{equation}
\end{definition}

Note that $v-d_P\al_P(v)$ cannot be negative. In fact, it is straightforward to verify $\ka(\C(\om))\in W_P$ for each $\om \in H_P$. Observe also that $0\leq \be_P(v) <p$.

Our goal is to show that in certain cases the coefficient of $\ka(\C(\om))$ is non-zero in $\C(\om)$. The following lemma gives a formula for this coefficient.

\begin{lemma} \label{lemktcom}
Let $\om=y_2^{a_2}y_1^{a_1}x_P^v dx \in H_P$. Then the coefficient of $\ka(\C(\om))$ in $\C(\om)$ is 
\begin{equation} \label{eqc_om}
    c_\om:= \begin{cases}
    (-1)^{\be_P(v)} u_P^{\be_P(v)/p} \binom{a_1}{\be_P(v)} &\hbox{\textup{if} $a_1 \geq \be_P(v)$} \\
    u_P^{\be_P(v)/p} \sum_{i=\be_P(v)-a_1}^{p-1-a_1} (-1)^{\be_P(v)+i+1} \frac{(p-1)!}{i!(p-i)!} \binom{a_1+i}{\be_P(v)} &\hbox{\textup{if} $a_1<\be_P(v)$.}
    \end{cases}
\end{equation}
\end{lemma}

\begin{proof}
By Equation~\eqref{eqCW}, we obtain
$$ \C(\om)= \sum_{0 \leq j+l \leq a_2} (-1)^{a_2-j} \binom{a_2}{j \, l} y_2^j \C(g(y_1)^l h(x)^{a_2-j-l} y_1^{a_1} x_P^v dx).$$
We split the proof into the two cases of Definition~\ref{defkeyterm}: $a_1 \geq \be(v)$ and $a_1<\be(v)$.

First, suppose $a_1 \geq \be(v)$. To obtain the exponent of $y_2$ in $\ka(\C(\om))$, we must specialize to $j=a_2$ in Equation~\eqref{eqCW}, yielding the term
$$ y_2^{a_2} \C(y_1^{a_1} x_P^v dx) = \sum_{m=0}^{a_1} (-1)^m \binom{a_1}{m} y_2^{a_2} y_1^{a_1-m} \C(f(x)^m x_P^v dx).$$
For the right exponent of $y_1$, we specialize to $m=\be_P(v)$. Then the leading term of $f_P(x_P)$ (recall Equation~\eqref{eqf=sumf_P}) gives the term
\begin{align*}
    \C(u_P^{\be_P(v)}x_P^{d_P\be_P(v)+v}dx) &= 
    u_P^{\be_P(v)/p} \C(x_P^{d_P(\g_P(v-\epsilon_P)-p\al_P(v))+v}dx) \\
    &= u_P^{\be_P(v)/p} \C(x_P^{(p-1)(v-\epsilon_P)+v-pd_P\al_P(v)}dx) \\
    &= u_P^{\be_P(v)/p} \C(x_P^{p(v-d_P\al_P(v)-\epsilon_P)+\epsilon_P}dx) \\
    &= u_P^{\be_P(v)/p} x_P^{v-d_P\al_P(v)}dx.
\end{align*}
This term is exactly $\ka(\C(\om))$. It is clear that this is the only contribution in $\C(\om)$ to $\ka(\C(\om))$. Therefore the coefficient of $\ka(\C(\om))$ in this case is
$$(-1)^{\be(v)} u_P^{\be_P(v)/p} \binom{a_1}{\be(v)}.$$

Now let us suppose $a_1 < \be(v)$. In that case we must set $j=a_2-1$ in Equation~\eqref{eqCW} to obtain the right exponent of $y_2$. Note that $a_2 \geq 1$ since $\om$ is an element of $H_P$. Now that $j=a_2-1$ has been specified, $l$ can still be $0$ or $1$, giving two terms:
$$y_2^{a_2-1} \left(-\C(y_1^{a_1}h(x)x_P^vdx) - \C(g(y_1)y_1^{a_1} x_P^v dx)  \right). $$
Since $a_1<a_1-\be(v)+p$ and $\C$ cannot increase the exponent of $y_1$, the first term cannot contribute to $\ka(\C(\om))$. For the second term, recall Equation~\eqref{eqdefg} for $g(y_1)$. We substitute to get
\begin{align*} 
&\sum_{i=1}^{p-1} (-1)^{i+1} \frac{(p-1)!}{i!(p-i)!} \C(y_1^{p(p-i)+i+a_1}x_P^vdx) \\
=  &\sum_{i=1}^{p-1} (-1)^{i+1} \frac{(p-1)!}{i!(p-i)!} y_1^{p-i} \C(y_1^{i+a_1}x_P^vdx) \\
= &\sum_{i=1}^{p-1} (-1)^{i+1} \frac{(p-1)!}{i!(p-i)!} \sum_{m=0}^{i+a_1} (-1)^m \binom{a_1+i}{m} y_1^{p-m+a_1} \C(f(x)^m x_P^v dx).
\end{align*}
For the exponent of $y_1$, we must set $m=\be_P(v)$. For the binomial coefficients to be non-zero, we must have $\be_P(v)-a_1 \leq i \leq p-1-a_1$. This choice of $m$ and the leading term of $f_P(x_P)$ yields the term
$$\sum_{i=\be_P(v)-a_1}^{p-1-a_1} \frac{(p-1)!}{i!(p-i)!} (-1)^{\be_P(v)+i+1} \binom{a_1+i}{\be_P(v)} y_2^{a_2-1} y_1^{a_1-\be(v)+p} \C(u_P^{\be_P(v)} x_P^{d_P\be(v)+v} dx).$$
As before, $\C(x_P^{d_P\be(v)+v}dx)=x^{v-d_P\al(v)},$ which finishes the proof that this is a contribution to $\ka(\C(\om))$. Again, one goes through the same steps to show that there is no other contribution to $\ka(\C(\om))$.
\end{proof}

To make future calculations easier, we give a sufficient criterion for $c_\om$ to be non-zero.

\begin{lemma} \label{lemcomnon-zero}
Let $P \in B_1$ and let $\om = y_2^{a_2}y_1^{a_1}x_P^vdx$ be a differential in $H_P$. If $\be_P(v) \leq a_1$ or $\be_P(v) \geq p-2$, then $c_\om$ is non-zero.
\end{lemma}

\begin{proof}
If $\be_P(v)\leq a_1$, it is clear from Lemma~\ref{lemktcom} that $c_\om \neq 0$. If $\be_P(v)=p-1$, then $c_\om$ is a sum consisting of only one non-zero term, so again $c_\om \neq 0$.

In the case $\be_P(v) = p-2$, $c_\om$ is a sum with two terms. To be precise, $c_\om$ vanishes if and only if 
$$ \sum_{i=p-2-a_1}^{p-1-a_1} \frac{(-1)^i (a_1+i)!}{i!(p-i)!\be_P(v)!(a_1+i-\be_P(v))!} = 0.$$
This sum vanishes precisely when its two terms are each other's negative:
\begin{align*}
&\frac{(p-2)!}{(p-2-a_1)!(a_1+2)!\be_P(v)!(p-2-\be_P(v))!}  \\ 
= &\frac{(p-1)!}{(p-1-a_1)!(a_1+1)!\be_P(v)!(p-1-\be_P(v))!}, \intertext{whence}
&(p-2)!(p-1-a_1)!(a_1+1)!\be_P(v)!(p-1-\be_P(v))! \\
= &(p-1)!(p-2-a_1)!(a_1+2)!\be_P(v)!(p-2-\be_P(v))!.
\end{align*}
Using the assumption $\be_P(v)=p-2$ yields
$$(p-2-\be_P(v))!= (p-1-\be_P(v))! = 1.$$
Then dividing both sides by $(p-2)!(p-2-a_1)!(a_1+1)\be_P(v)!$ gives
\begin{align*}
(p-1-a_1) &= (p-1)(a_1+2) \\
-a_1 -1 &= -(a_1+2) \\
-1 &= -2.
\end{align*}
From this contradiction we infer that $c_\om \neq 0$ in the case $\be_P(v)=p-2$.
\end{proof}

\subsection{Key terms at poles of \texorpdfstring{$h$}{}} \label{subsec:ktpoleh}

Let $B_2 \subset \PP^1$ denote the branch locus of the cover $\phi_2: Y_2 \to \PP^1$. We consider a point $P \in B_2 \setminus B_1$. This implies that $f$ is regular at $P$ but $h$ has a pole at $P$. We denote the order of this pole by $e_P$. By minimality, we have $e_P | p-1$ and define $\delta_P = \frac{p-1}{e_P}$.

Proposition~\ref{propMadden} yields the set
\begin{equation}
\begin{aligned} \label{eqdefWPpoleh}
    W_P:= \begin{cases} y_2^{a_2} y_1^{a_1} x_P^v dx \; | \; 0 \leq a_1,a_2 <p, \\  
       0 < p v \leq  e_P(p-1-a_2) +p-1 \end{cases}  \Bigg\}.
\end{aligned}
\end{equation}
For a differential $\om = y_2^{a_2} y_1^{a_1} x_P^v dx \in W_P$, Lemma~\ref{lemC(om)} holds, as the proof is independent of whether or not $f$ has a pole at $P$. This allows us to define the key term of $\om$.

\begin{definition} \label{defktpoleh}
Let $\alpha_p(v):= \left \lfloor \frac{\delta_P(v-1)}{p} \right \rfloor$ and $\be_P(v):=\delta_P(v-1)-p\al_P(v)$. Define
$$H_P:= \{ y_2^{a_2} y_1^{a_1} x_P^v dx \in W_P \; | \; \be_P(v) \leq a_2 \}$$
For $\om := y_2^{a_2} y_1^{a_1} x_P^v dx \in H_P$ we define the associated key term $\ka(\C(\om))$ as
\begin{equation*} \label{eqktpoleh}
\ka(\C(\om)) := y_2^{a_2-\be_P(v)}y_1^{a_1} x_P^{v-e_P\al_P(v)} dx.
\end{equation*}
\end{definition}

\begin{lemma} \label{lemcompoleh}
Let $\om = y_2^{a_2} y_1^{a_1} x_P^v \in H_P$. Then the coefficient $c_\om$ of $\ka(\C(\om))$ in $\C(\om)$ is non-zero.
\end{lemma}
\begin{proof}
Recall Equation~\eqref{eqCW}:
$$\C(\om)= \sum_{0 \leq j+l \leq a_2} (-1)^{a_2-j} \binom{a_2}{j \, l} y_2^j \C(g(y_1)^l h(x)^{a_2-j-l} y_1^{a_1} x_P^v dx).$$
In order to obtain the exponent $a_2-\be_p(v)$ of $y_2$, we must pick $j=a_2-\be_P(v)$. Picking $l=0$ results in a term 
$$y_2^{a_2-\be_P(v)}\C(y_1^{a_1} h(x)^{\be_P(v)} x_P^vdx)= \sum_{m=0}^{a_1} (-1)^m \binom{a_1}{m} y_2^{a_2-\be_P(v)} y_1^{a_1-m} \C(h(x)^{\be_P(v)} f(x)^m x_P^v dx).$$
We pick $m=0$. The differential $\C(h(x)^{\be_P(v)} x_P^v dx)$ has a non-zero term
\begin{align*}
\C(x_P^{e_P\be_P(v) + v} dx) &= \C(x_P^{e_P(\delta_P(v-1) - p\al_P(v)) + v} dx) \\
&= x_P^{v-e_P\al_P(v)} dx.
\end{align*}
We do not get such a term if we pick $l>0$. This shows that there is a single non-zero contribution to the coefficient of $\ka(\C(\om))$ in $\C(\om)$, implying that the coefficient is non-zero. 
\end{proof}

\subsection{An upper bound on the \texorpdfstring{$a$}{}-number}

We now use the machinery of key terms to provide a conditional upper bound on the $a$-number $a_{Y_2}$. We do this by showing the Cartier operator is injective when restricted to a suitable subspace of $\cohom^0(Y_2, \Om_{Y_2}^1)$, whose dimension bounds the rank of the Cartier operator from below.

First we define $W := \bigcup_{P \in B_2} W_P$, such that $W$ is a basis of $\cohom^0(Y_2,\Om_{Y_2}^1)$. Let $H := \bigcup_{P \in B_2} H_P$ be the subset of $W$ consisting of differentials that have a key term. Finally, let $K$ be the set of key terms. As opposed to the case described in \cite{MR3090462}, it is now possible for two differentials to have the same key term. The following lemma describes exactly when this happens.

\begin{lemma} \label{lemsamekeyterms}
If two differentials $\om:=y_2^{a_2}y_1^{a_1}x_P^vdx \in H_P$ and $\om':=y_2^{b_2}y_1^{b_1}x_Q^wdx \in H_Q$ have the same key term, then $P=Q$. In the case $P \in B_1$, then $\om$ and $\om'$ have the same key term if and only if the following equations are satisfied:  
\begin{align} 
v+d_P(a_1+pa_2) &= w+d_P(b_1+pb_2) \label{eqsamekeyterms} \\
v-d_P\al_P(v) &= w-d_P\al_P(w). \label{eqsamekeyterms2}
\end{align}
In the case $P \in B_2 \setminus B_1$, $\om$ and $\om'$ have the same key term if and only if the following equations are satisfied:
\begin{align} 
v+e_P a_2 &= w+e_P b_2 \label{eqsamekeytermsh} \\
a_1 &= b_1 \label{eqsamekeytermsy1}\\
v-e_P\al_P(v) &= w-e_P\al_P(w). \label{eqsamekeyterms2h}
\end{align}
\end{lemma}
\begin{proof}
Assume $\ka(\C(\om))=\ka(\C(\om'))$. First we show that $P=Q$. Note that, if $P\neq Q$ then the key terms of $\om$ and $\om'$ must not involve $x_P$ and $x_Q$, such that  
$$ v-d_P\al_P(v) = w-d_Q\al_Q(v)=0.$$
However, for $P \neq \infty$ we have
$$ v-d_P\al_P(v) = v-d_P \left\lfloor \frac{\g_P(v-1)}{p} \right\rfloor \geq v-\frac{d_P\g_P (v-1)}{p} = v- \frac{(p-1)(v-1)}{p}>0.$$
This implies that $P=Q=\infty$ and hence $P$ and $Q$ are not distinct.

Assume $P \in B_1$. We show that Equations~\eqref{eqsamekeyterms} and \eqref{eqsamekeyterms2} hold. Equating the exponents of $x_P$ in $\ka(\C(\om))$ and $\ka(\C(\om'))$ immediately gives Equation~\eqref{eqsamekeyterms2}, so the rest of the proof revolves around showing Equation~\eqref{eqsamekeyterms}.
By equating the exponent of $y_2$ in the key terms, we obtain $b_2 \in \{a_2-1,a_2,a_2+1\}$. Without loss of generality, we assume $b_2 \geq a_2$, such that only the cases $b_2=a_2$ and $b_2=a_2+1$ remain. 

In the case $b_2=a_2$, equating the exponent of $x_P$ in the key terms gives
\begin{equation} \label{eqsamekeytermsvw}
    v= w+d_P(\al_P(v)-\al_P(w)).
\end{equation}
Substituting this into the equation for the exponents of $y_1$ yields
\begin{align*}
    a_1-\be_P(v) &= b_1 - \be_P(w) \\
    a_1-\g_P(v-\epsilon_P)+p\al_P(v)&=b_1-\g_P(w-\epsilon_P)+p\al_P(w) \\
    a_1-\g_P(w+d_P(\al_P(v)-\al_P(w)-\epsilon_P)+p\al_P(v) &= b_1-\g_P(w-\epsilon_P)+p\al_P(w) \\
    a_1-(p-1)(\al_P(v)-\al_P(w))+p\al_P(v) &= b_1+p\al_P(w) \\
    \al_P(v)-\al_P(w)&=b_1-a_1.
\end{align*}
Substituting this back into Equation~\eqref{eqsamekeytermsvw} then gives $v=w+d_P(b_1-a_1)$ and hence $v+d_Pa_1=w+d_Pb_1$, as desired.

On the other hand, suppose $b_2=a_2+1$. In that case, Equation~\eqref{eqsamekeytermsvw} still follows from equating the exponent of $x_P$. Equating the exponent of $y_1$ yields $a_1-\be_P(v) = b_1-\be_P(w).$ Following the same steps as in the case $b_2=a_2$ gives $\al_P(v)-\al_P(w)=b_1-a_1+p$ and hence
$v=w+d_P(b_1-a_1+p)$. Rearranging this and using $b_2-a_2=1$ gives the desired Equation~\eqref{eqsamekeyterms}.

For the opposite direction, assume Equations \eqref{eqsamekeyterms} and \eqref{eqsamekeyterms2} hold. By Equation~\eqref{eqsamekeyterms2}, $\ka(\C(\om))$ and $\ka(\C(\om'))$ have the same exponent of $x_P$. Substituting $v=w+d_P(\al_P(v)-\al_P(w))$ in Equation~\eqref{eqsamekeyterms} yields
$$\al_P(v)-\al_P(w)=b_1-a_1+p(b_2-a_2).$$ 
This implies
\begin{align*}
    b_1-\be_P(w)+p(b_2-a_2)&=a_1+\al_P(v)-\al_P(w)-\be_P(w) \\
    &=a_1+\al_P(v)-\al_P(w)-\g_P(w-\epsilon_P)+p\al_P(w) \\
    &=a_1+\al_P(v)-\al_P(w) \\ 
     & \hspace{5mm} -\g_P(v+d_P(\al_P(w)-\al_P(v))-\epsilon_P)+p\al_P(w) \\
    &=a_1+\al_P(v)-\al_P(w)-\g_P(v-\epsilon_P) \\
    & \hspace{5mm} -(p-1)(\al_P(w)-\al_P(v))+p\al_P(w) \\
    &=a_1-\g_P(v-\epsilon_P)+p\al_P(v) \\
    &=a_1-\be_P(v).
\end{align*}
Using this, one verifies that in all cases for $b_2$, the exponents of $y_1$ and $y_2$ in $\ka(\C(\om))$ and $\ka(\C(\om'))$ agree. Thus $\om$ and $\om'$ have the same key term.

We now consider the case $P \in B_2 \setminus B_1$. Equation~\eqref{eqsamekeyterms2h} follows immediately from equating the exponent of $x_P$ in the key term. Similarly, Equation~\eqref{eqsamekeytermsy1} follows from equating the exponents of $y_1$. For Equation~\eqref{eqsamekeytermsh}, we write $v=w+e_Pr$ and $\al_P(v)=\al_P(w)-r$ for some integer $r$. By equating the exponents of $y_2$ in the key term, we obtain
\begin{align*}
    a_2-b_2 &= \be_P(v)-\be_P(w) \\
    &= \delta_P(v-1)-p\al_P(v) - (\delta(w-1) - p \al_P(w)) \\
    &= \delta_P(v-w) + p (\al_P(v)-\al_P(w)) \\
    &= \delta_P e_P r + pr = -r.
\end{align*}
Thus Equation~\eqref{eqsamekeytermsh} also holds.

For the opposite direction, assume Equations \eqref{eqsamekeytermsh}, \eqref{eqsamekeytermsy1} and \eqref{eqsamekeyterms2h} hold. Then we only need to show that $y_2$ has the same exponent in both key terms. This is achieved by writing $a_2=b_2-r$ and $v=w+e_Pr$. Then a similar calculation as above shows $$a_2-\beta_P(v) = b_2- \be_P(w),$$ as desired.
\end{proof}

To display a minor of the Cartier-Manin matrix in a convenient form, we define a partial order on $W$. First, let $<$ be any linear order on $B_2$ such that $\infty<P$ for each $P\neq \infty$.

\begin{definition} \label{defprec}
Let $\om:=y_2^{a_2} y_1^{a_1} x_P^v dx$ and $\om':=y_2^{b_2} y_1^{b_1} x_Q^w dx$. We say $\om' \prec \om$ in any of the following cases:
\begin{enumerate}[label=(\roman*)]
    \item $Q\neq P$ and $b_1+pb_2 < a_1+pa_2$. \label{defpreccasePneqQ1}
    \item $Q<P$ and $b_1+pb_2=a_1+pa_2$. \label{defpreccasePneqQ2}
    \item $Q=P \in B_1$ and $w + d_P(b_1+pb_2) < v+ d_P(a_1+pa_2)$. \label{defpreccaseP=Q1}
    \item $Q=P \in B_1$ and $w + d_P(b_1+pb_2) = v+ d_P(a_1+pa_2)$ and $w-d_P\al_P(w) > v-d_P\al_P(v)$. \label{defpreccaseP=Q2}
    \item $Q=P \in B_2 \setminus B_1$ and $w + e_P a_2 < v+ e_Pa_2$. \label{defpreccaseP=Qh1}
    \item $Q=P \in B_2 \setminus B_1$ and $w + e_P b_2 = v+e_P a_2$ and $b_1 < a_1$. \label{defpreccaseP=Qh2}
    \item $Q=P \in B_2 \setminus B_1$ and $w + e_P b_2 = c+e_P a_2$ and $b_1 = a_1$ and $w-e_P\al_P(w) > v-d_P\al_P(v)$. \label{defpreccaseP=Qh3}
\end{enumerate}
\end{definition}

\begin{remark} \label{rmkincomparableiffsamekeyterm}
By comparing Definition~\ref{defprec} to Lemma~\ref{lemsamekeyterms}, note that two differentials $\om$ and $\om'$ are incomparable by the order $\prec$ if and only if they have the same key term.
\end{remark}

The goal of introducing the partial order $\prec$ is to prove the following lemma.

\begin{lemma} \label{lemktzero}
Let $\om:=y_2^{a_2} y_1^{a_1} x_P^v dx \in H$ and $\om':=y_2^{b_2} y_1^{b_1} x_Q^w dx \in W$ with $\om' \prec \om$. Then the coefficient of the key term $\ka(\C(\om))$ is zero in $\C(\om')$.
\end{lemma}

In order to prove Lemma~\ref{lemktzero}, we prove seven lemmas, corresponding to the cases of Definition~\ref{defprec}. In each case, we have
\begin{equation} \label{eqC(om')ktzero}
\C(\om')= \sum_{0 \leq j+l \leq b_2} (-1)^{b_2-j} \binom{b_2}{j \, l} y_2^j \C(g(y_1)^l h(x)^{b_2-j-l} y_1^{b_1} x_Q^w dx).
\end{equation}

Recall also the formula for the key term $\ka(\C(\om)):$
\begin{equation} \label{eqkeytermprec}
\ka(\C(\om)):= \begin{cases} 
y_2^{a_2}y_1^{a_1-\be_P(v)}x_P^{v-d_P\al_P(v)} dx  &\hbox{\textup{if } $P \in B_1$ \textup{ and } $a_1 \geq \be_P(v)$} \\
y_2^{a_2-1}y_1^{a_1-\be_P(v)+p}x_P^{v-d_P\al_P(v)} dx &\hbox{\textup{if } $P \in B_1$ \textup{ and } $a_1 < \be_P(v)$} \\
y_2^{a_2-\be_P(v)}y_1^{a_1}x_P^{v-e_P \al_P(v)} dx &\hbox{\textup{if } $P \in B_2 \setminus B_1$.}
\end{cases}
\end{equation}

The general strategy for proving Lemma~\ref{lemktzeroPneqQ1} to Lemma~\ref{lemktzeroP=Qh3} is as follows. We assume there is a term of $\mathcal{C}(\om')$ that contributes to $\ka(\C(\om))$. Then we specialize to the terms of $\mathcal{C}(\om')$ that have the required exponent of $y_2$ using Equation~\eqref{eqkeytermprec}. We then specialize to the terms that have the required exponent of $y_1$. This way we are left with a differential on $\mathbb{P}^1$ and we show that its pole order is too small to give the required exponent of $x_P$. The details of this process differ per case of Definition~\ref{defprec}.

\begin{lemma}[Case \ref{defpreccasePneqQ1} of Definition~\ref{defprec}] \label{lemktzeroPneqQ1}
Let $\om:=y_2^{a_2} y_1^{a_1} x_P^v dx \in H$ and $\om':=y_2^{b_2} y_1^{b_1} x_Q^w dx \in W$ with $Q \neq P$ and $b_1+pb_2 < a_1 + pa_2$. Then the coefficient of $\ka(\C(\om))$ is zero in $\C(\om')$.
\end{lemma}
\begin{proof}
Since $0 \leq a_1,b_1 < p$, the assumption implies $b_2 \leq a_2$. First assume that $P \in B_1$. Then it is clear from Equation~\eqref{eqC(om')ktzero} and Equation~\eqref{eqkeytermprec} that we are done in the case $b_2<a_2-1$. In the case $b_2=a_2-1$, we must have $a_1 < \be_P(v)$. In Equation~\eqref{eqC(om')ktzero}, we need to specialize to $j=a_2-1$ and $l=0$ to get the right exponent of $y_2$. This leaves us with a term
$$ y_2^{a_2-1} \C(y_1^{b_1} x_Q^w dx) = \sum_{m=1}^{b_1} y_2^{a_2-1} y_1^{b_1-m} \C(f(x)^m x_Q^w dx).$$
To get the right exponent of $y_1$, we must pick $m=b_1-a_1+\be_P(v)-p$. Then the differential $f(x)^{b_1-a_1+\be_P(v)-p}x_Q^w dx$ has a pole of order at most $d_P(b_1-a_1+\be_P(v)-p)$ at $P$. However, a pole of order $d_P\be_P(v)+v$ is needed to produce the exponent $v-d_P\al_P(v)$ of $x_P$. Since $b_1-a_1-p <0$, it follows that this exponent is not attained and hence $\ka(\C(\om))$ is zero in $\C(\om')$.

For $b_2=a_2$ the proof is similar. In the case $a_1 \geq \be_P(v)$, we must pick $j=a_2$ and $l=0$ in Equation~\eqref{eqC(om')ktzero} in order to get the right exponent of $y_2$. This yields the term
$$ y_2^{a_2} \C(y_1^{b_1} x_Q^w dx) = \sum_{m=0}^{b_1} (-1)^m \binom{b_1}{m} y_2^{a_2} y_1^{b_1-m} \C(f(x)^m x_Q^w dx).$$
In order to achieve the right exponent of $y_1$, we must specialize $m=b_1-a_1+\be_P(v)$. Since $\om'\prec \om$, we have $b_1 < a_1$ and it follows that the right exponent of $x_P$ is not achieved. In the case $a_1 < \be_P(v)$, we must specialize to $j=a_2-1$ in Equation~\eqref{eqC(om')ktzero} to obtain the right exponent of $y_2$. The possibilities $l=0$ and $l=1$ result in the terms $y_2^{a_2-1} \C(g(y_1) y_1^{b_1} x_Q^w dx)$ and $y_2^{a_2-1} \C(y_1^{b_1} h(x) x_Q^{w} dx).$
The second term could never attain the right exponent of $y_1$, since $b_1<a_1$. On the other hand, for the first term we substitute the definition of $g(y_1)$ to obtain
\begin{align*}
    \C(g(y_1) y_1^{b_1}x_Q^w dx) &= \sum_{i=1}^{p-1} (-1)^{i} \frac{(p-1)!}{i!(p-i)!} y_1^{p-i} \C(y_1^{b_1+i} x_Q^w dx) \\
    &= \sum_{i=1}^{p-1} \frac{(p-1)!}{i!(p-i)!} \sum_{m=0}^{b_1+i} (-1)^{m+i+1} \binom{b_1+i}{m} y_1^{p-m+b_1} \C(f(x)^m x_Q^w dx).
\end{align*}
For any $i$, we must pick $m=b_1-a_1+\be_P(v)$ to obtain the right exponent of $y_1$. Then the differential $f(x)^{b_1-a_1+\be_P(v)} x_Q^w dx$ has a pole of order at most $d_P(b_1-a_1+\be_P(v))$ at $P$. Since $b_1 < a_1$, this is smaller than the required pole order $d_P\be_P(v)+v$ to obtain the right exponent of $x_P$. Hence the coefficient of $\ka(\C(\om))$ is again zero in $\C(\om')$.

Now, assume that $P \in B_2 \setminus B_1$. In Equation~\eqref{eqC(om')ktzero}, we must pick $j=a_2-\be_P(v)$ to get the right exponent of $y_2$. In the case $b_1 < a_1$, the required exponent of $y_1$ can never be obtained. On the other hand, in the case $b_2 < a_2$, the required exponent $\be_P(v)$ of $h(x)$ cannot be obtained. In both cases, the coefficient of $y_2^{a_2-\be_P(v)}y_1^{a_1}x_P^{v-e_P\al_P(v)}dx$ in $\C(\om')$ is zero, which finishes the proof.
\end{proof}

\begin{lemma}[Case \ref{defpreccasePneqQ2} of Definition \ref{defprec}] \label{lemktzeroPneqQ2}
Let $\om:=y_2^{a_2} y_1^{a_1} x_P^v dx \in H$ and $\om':=y_2^{b_2} y_1^{b_1} x_Q^w dx \in W$ with $Q<P$ and $b_1+pb_2=a_1+pa_2$. Then the coefficient of $\ka(\C(\om))$ is zero in $\C(\om')$.
\end{lemma}
\begin{proof}
The proof is similar to the proof of Lemma~\ref{lemktzeroPneqQ1}. Note that we have $b_2=a_2$ and $b_1=a_1$. First suppose $P \in B_1$. After attaining the right exponents of $y_2$ and $y_1$, we are left with a differential on $\PP^1$ that has a pole of order at most $d_P\be_P(v)$ at $P$, whereas a pole of order $d_P\be_P(v)+v$ is needed. Then, it follows from the assumption $Q<P$ that $P \neq \infty$ (since $\infty$ is the smallest branch point in our ordering) and therefore $v>0$. Thus the required pole order is not attained and therefore the coefficient of $\ka(\C(\om))$ is zero in $\C(\om')$.

Similarly, suppose $P \in B_2 \setminus B_1$. In Equation~\eqref{eqC(om')ktzero}, we must choose $j=a_2-\be_P(v)$, such that the resulting differential of $\PP^1$ has a pole of order at most $e_P\be_P(v)$, whereas a pole of order $e_P \be_P(v)+v$ is needed. Again, noting $v>0$ finishes the proof.
\end{proof}

\begin{lemma}[Case \ref{defpreccaseP=Q1} of Definition \ref{defprec}] \label{lemktzeroP=Q1}
Let $\om:=y_2^{a_2} y_1^{a_1} x_P^v dx \in H$ and $\om':=y_2^{b_2} y_1^{b_1} x_P^w dx \in W$ with $P \in B_1$ and \begin{equation} \label{eqprecP=Q1}
w+d_P(b_1+pb_2) < v+d_P(a_1+pa_2). 
\end{equation}
Then the coefficient of $\ka(\C(\om))$ is zero in $\C(\om')$.
\end{lemma} 
\begin{proof}
The same strategy as before is used. Equation~\eqref{eqprecP=Q1} and Equation~\eqref{eqv<pd1} together imply that $b_2 \leq a_2+1$. Moreover, by Equations \eqref{eqC(om')ktzero} and \eqref{eqkeytermprec} it follows that we are done if $b_2 < a_2-1$. We split the rest of the proof up into the two possibilities for $\ka(\C(\om))$: either $a_1 \geq \be_P(v)$ or $a_1 < \be_P(v)$.

In the case $a_1 \geq \be_P(v)$, we have
$$\ka(\C(\om))= y_2^{a_2} y_1^{a_1-\be_P(v)} x^{v-d_P\al_P(v)} dx.$$
This forces $b_2 \geq a_2$, so $b_2$ equals either $a_2$ or $a_2+1$. In the case $b_2=a_2$, possible contributions to $\ka(\C(\om))$ of $\C(\om')$ come from setting $j=a_2$ and $l=0$ in Equation~\eqref{eqC(om')ktzero}. This leaves the term
$$ y_2^{a_2} \C(y_1^{b_1} x_Q^w dx) = \sum_{m=0}^{b_1} (-1)^{m+1} \binom{b_1}{m} y_2^{a_2} y_1^{b_1-m} \C(f(x)^m x_P^w dx).$$
For the right exponent of $y_1$, we need $m=b_1-a_1+\be_P(v)$. This results in the term $\C(f(x)^{b_1-a_1+\be_P(v)} x_P^w)$, which has a pole of order $d_P(b_1-a_1+ \be_P(v))+w$, which is smaller than the required $d_P\be_P(v)+v$ by Equation~\eqref{eqprecP=Q1}. The case $b_2=a_2+1$ is handled similarly. In Equation~\eqref{eqC(om')ktzero}, possible contributions come from $j=a_2$ and $l\in\{0,1\}$. In both cases for $l$, the resulting differential $f(x)^m x_P^w dx$ does not have sufficiently large pole order at $P$ to contribute to $\ka(\C(\om))$.

In the case $a_1 < \be_P(v)$, we have
$$\ka(\C(\om))= y_2^{a_2-1} y_1^{a_1-\be_P(v)+p} x_P^{v-d_P\al_P(v)} dx.$$
This leaves three possibilities for $b_2$: $a_2-1$, $a_2$ and $a_2+1$. We treat only the case $b_2=a_2$, as the other cases are similar. In Equation~\eqref{eqC(om')ktzero} we must specialize to $j=a_2-1$. The possibilities $l=0$ and $l=1$ give the two terms $y_2^{a_2-1} \C(g(y_1) y_1^{b_1} x_P^w dx)$ and $y_2^{a_2-1} \C(y_1^{b_1} h(x) x_P^w dx).$ For the first term, we expand
\begin{align*}
    \C(g(y_1) y_1^{b_1}x_P^w dx) &= \sum_{i=1}^{p-1} (-1)^{i} \frac{(p-1)!}{i!(p-i)!} y_1^{p-i} \C(y_1^{b_1+i} x_P^w dx) \\
    &= \sum_{i=1}^{p-1} \frac{(p-1)!}{i!(p-i)!} \sum_{m=0}^{b_1+i} (-1)^{m+i+1} \binom{b_1+i}{m} y_1^{p-m+b_1} \C(f(x)^m x_P^w dx).
\end{align*}
To obtain the right exponent of $y_1$, setting $m=b_1-a_1+\be_P(v)$ is needed. Then Equation~\eqref{eqprecP=Q1} implies that $\C(f(x)^m x_P^w dx)$ does not have a sufficiently large pole order at $P$. For the second term, we expand
$$\C(y_1^{b_1} h(x) x_P^w dx) = \sum_{m=1}^{b_1} (-1)^{i} \binom{b_1}{m} y_1^{b_1-m} \C(h(x) f(x)^m x_P^w dx).$$
The right exponent of $y_1$ results only from $m=b_1-a_1+\be_P(v)-p$. Then since the pole order of $h(x)$ at $P$ is smaller than $pd_P$ (this is implied by the assumption that the upper jumps are minimal), the differential $h(x) f(x)^m x_P^w dx$ has pole order smaller than $d_P(b_1-a_1+\be_P(v))+w$. By Equation~\eqref{eqprecP=Q1}, this is smaller than the required $v + d_P\be_P(v)$. This implies that the coefficient of $\ka(\C(\om))$ is zero in $\C(\om')$. 
\end{proof}

\begin{lemma}[Case \ref{defpreccaseP=Q2} of Definition \ref{defprec}] \label{lemktzeroP=Q2}
Let $\om:=y_2^{a_2} y_1^{a_1} x_P^v dx \in H$ and $\om':=y_2^{b_2} y_1^{b_1} x_P^w dx \in W$ with $P \in B_1$ and $w+d_P(b_1+pb_2) = v + d_P (a_1+pa_2)$ and $w-d_P\al_P(w) > v-d_P\al_P(v)$. Then the coefficient of $\ka(\C(\om))$ is zero in $\C(\om')$.
\end{lemma}
\begin{proof}
From the first condition, we infer $w \equiv v \bmod {d_P}$, so we write $w=v+ed_P$ and $b_1+pb_2=a_1+pa_2-e$. By the second condition, we must have 
\begin{align*}
    \al_P(w)&=\al_P(v+ed_P)=\al_P(v)+e-1. \intertext{This implies} 
    \beta_P(w)&=\be_P(v+ed_P)=\be_P(v)-e+p.
\end{align*}
Note that this forces $e>\be_P(v)$. We have
\begin{equation} \label{eqom'P=Q2}
\omega' = \begin{cases}
y_2^{a_2} y_1^{a_1-e} x_P^{v+ed_P} dx &\hbox{if $a_1 \geq e$} \\
y_2^{a_2-1} y_1^{a_1-e+p} x_P^{v+ed_P} dx &\hbox{if $a_1<e$}.
\end{cases} 
\end{equation}
We again consider the two possibilities determining the key term of $\om$: $a_1 \geq \be_P(v)$ and $a_1<\be_P(v)$. 

We treat the case $a_1 \geq \be_P(v)$ first, in which the key term is
$$\ka(\C(\om))= y_2^{a_2} y_1^{a_1-\be_P(v)} x_P^{v-d_P\al_P(v)} dx.$$
To achieve this as a term in $\C(\om')$, we need $b_2 \geq a_2$. Appealing to Equation~\eqref{eqom'P=Q2}, only the case $a_1 \geq e$ remains. In Equation~\eqref{eqC(om')ktzero} we need $j=a_2$ and $l=0$ to obtain the required exponent of $y_2$. This gives the term
$$y_2^{a_2} \C(y_1^{a_1-e} x_P^{v+ed_P} dx) = \sum_{m=0}^{a_1-e} (-1)^m \binom{a_1-e}{m} y_2^{a_2} y_1^{a_1-e-m} \C(f(x)^m x_P^w dx).$$
In order to obtain the exponent of $y_1$, setting $m=\be_P(v)-e$ is required, but this is negative as we have ascertained $e > \be_P(v)$. Thus the right exponent of $y_1$ cannot be achieved.

On the other hand, assume $a_1 < \be_P(v)$. Note that this implies $a_1 < e$, since $\be_P(v)<e$. By Equation~\eqref{eqom'P=Q2}, this determines $\om'$. In Equation~\eqref{eqC(om')ktzero}, we have to specialize to $j=a_2-1$ and $l=0$, yielding the term
$$ y_2^{a_2-1} \C(y_1^{a_1-e+p} x_P^w dx) = \sum_{m=0}^{a_1-e+p} (-1)^m \binom{a_1-e+p}{m} y_2^{a_2-1} y_1^{a_1-e+p-m} \C(f(x)^m x_P^w dx).$$
Again the required exponent of $y_1$ can only be attained if $m=\be_P(v)-e$, which is negative. 
\end{proof}

\begin{lemma}[Case \ref{defpreccaseP=Qh1} of Definition \ref{defprec}] \label{lemktzeroP=Qh1}
Let $\om:=y_2^{a_2}y_1^{a_1}x_P^v dx\in H$ and $\om':=y_2^{b_2} y_1^{b_1}x_P^{w}$ with $P \in B_2 \setminus B_1$ and $w+e_P b_2 < v+e_Pa_2$. Then the coefficient of $\ka(\C(\om))$ is zero in $\C(\om')$.
\end{lemma}
\begin{proof}
In order to obtain the key term $\ka(\C(\om))$, we have to take $j=a_2-\be_P(v)$. Then the exponent of $h(x)$ can be at most $b_2-a_2+\be_P(v)$, making the exponent of $x_P$ at most $e_P(b_2-a_2+\be_P(v))+w$, which by our assumption is smaller than the required $e_P \be_P(v) + v$. 
\end{proof}

\begin{lemma}[Case \ref{defpreccaseP=Qh2} of Definition \ref{defprec}] \label{lemktzeroP=Qh2}
Let $\om:=y_2^{a_2}y_1^{a_1}x_P^v dx\in H$ and $\om':=y_2^{b_2} y_1^{b_1}x_P^{w}$ with $P \in B_2 \setminus B_1$, $w+e_P b_2 = v+e_Pa_2$ and $b_1 < a_1$. Then the coefficient of $\ka(\C(\om))$ is zero in $\C(\om')$.
\end{lemma}
\begin{proof}
In Equation~\eqref{eqC(om')ktzero}, we must take $j=a_2-\be_P(v)$ and $l=0$ to get the right exponents of $y_2$ and $x_P$. Then the differential 
$$\C(y_1^{b_1} h(x)^{b_2-a_2+\be_P(v)} x_P^w dx)$$
cannot have a term in which exponent of $y_1$ is $a_1$. This finishes the proof. 
\end{proof}

\begin{lemma}[Case \ref{defpreccaseP=Qh3} of Definition \ref{defprec}] \label{lemktzeroP=Qh3}
Let $\om:=y_2^{a_2}y_1^{a_1}x_P^v dx\in H$ and $\om':=y_2^{b_2} y_1^{b_1}x_P^{w}$ with $P \in B_2 \setminus B_1$, $w+e_P b_2 = v+e_Pa_2$, $b_1 = a_1$ and $w-e_P \al_P(w) > v-e_P \al_P(v)$. Then the coefficient of $\ka(\C(\om))$ is zero in $\C(\om')$.
\end{lemma}
\begin{proof}
The proof is similar to the proof of Lemma~\ref{lemktzeroP=Q2}. Write $b_2=a_2-r$ and $w=v+re_P$ for some integer $r$. From the inequality, we infer
\begin{align*}
    \al_P(w)&=\al_P(v+re_P)=\al_P(v)+r-1 \\
    \beta_P(w)&=\be_P(v+re_P)=\be_P(v)-r+p.
\end{align*}
This forces $r > \be_P(v)$. Then in Equation~\eqref{eqC(om')ktzero}, we need to take $j=a_2-\be_P(v)$ in order to obtain the right exponent of $y_2$. But this is impossible as it exceeds $b_2=a_2-r$. Therefore there can be no contribution to $\ka(\C(\om))$ in $\C(\om')$.
\end{proof}

Lemma~\ref{lemktzero} now follows.

\begin{proof}[Proof of Lemma~\ref{lemktzero}]
We have treated all seven cases of $\om' \prec \om$ given in Definition~\ref{defprec} in Lemmas \ref{lemktzeroPneqQ1}, \ref{lemktzeroPneqQ2}, \ref{lemktzeroP=Q1}, \ref{lemktzeroP=Q2}, \ref{lemktzeroP=Qh1}, \ref{lemktzeroP=Qh2} and \ref{lemktzeroP=Qh3}. In every case, the coefficient of $\ka(\C(\om))$ is zero in $\C(\om')$.
\end{proof}

We now finally use key terms to prove a lower bound on the rank of the Cartier operator. Recall that $K$ is the set of key terms and that $H \subset W$ is the set of basis regular differentials for which key terms are defined.

\begin{theorem} \label{thmrkCgeqK}
Assume that for each $\om' \in K$, there is an $\om\in H$ such that $\ka(\C(\om))=\om'$ and $c_\om \neq 0$. Then we have $\rk(\C) \geq \#K$.
\end{theorem}

\begin{proof}
By the assumptions, we can find a subset $H' \subseteq H$ such that every $\om'\in K$ has a unique $\om\in H'$ with the properties $\ka(\C(\om))=\om'$ and $c_\om \neq 0$. Note that $H'$ and $K$ have the same cardinality. By Lemma~\ref{lemsamekeyterms} and Remark \ref{rmkincomparableiffsamekeyterm}, the partial order $\prec$ is a linear order when restricted to $H'$. We restrict the Cartier operator to $\Span_k H'$ and then project to $\Span_k K$:
$$ \C |_{H'}: \Span_k H' \to \Span_k K.$$
By Lemma~\ref{lemktzero}, the matrix respresenting $C|_{H'}$ with respect to the basis $H'$ ordered by $\prec$ has only zeros below the diagonal. Furthermore, by Lemma~\ref{lemktcom}, the diagonal entries are $c_{\om}$ in the case $P \in B_1$. These are non-zero by assumption. In the case $P \in B_2 \setminus B_1$, $c_\om$ is non-zero unconditionally by Lemma~\ref{lemcompoleh}. Hence $\C|_{H'}$ is invertible, implying that $\rk(\C) \geq \#K.$
\end{proof}

\begin{remark} \label{rmk:imagenotspan}
Optimistically one might hope to prove the equality $\im(\C) = \Span_k K$. However, both inclusions fail already in the case $p=3$. See Remark~\ref{rmk:p=3imagenotkeyterm}.
\end{remark}

\section{The case \texorpdfstring{$p=3$}{}} \label{sec:p=3}

In this section we use the machinery from the preceding sections to prove Conjecture~\ref{conjnumberminimal} for $\Z/9\Z$-covers of $\PP^1$ in characteristic $3$. In Theorem~\ref{thmp=3}, we prove that the upper bound from Theorem~\ref{thmrkCgeqK} is sharp for $\Z/9\Z$ covers of $\PP^1$.

Let $\phi_2:Y_2 \to \PP^1$ be such a cover, with branching datum
\begin{equation*}
\label{eqnminimalp=3}
  \left[
\begin{array}{rrrrrrrr}
3&  \ldots  & 2 & \ldots & 0 & \ldots & 0 & \ldots \\ 
\undermat{n_1}{7 &  \ldots} & \undermat{n_2}{4 &  \ldots} &  \undermat{n_3}{3 &  \ldots}  &  \undermat{n_4}{{2} &  \ldots}\\ 
\end{array} 
\right]^{\intercal}.
\end{equation*}
\\
Recall the models
\begin{align}
    Y_1:y_1^3-y_1 &= f(x)  \\
    Y_2:y_2^3-y_2 &= g(y_1)+h(x),
\end{align}
where 
\begin{equation*}
g(y_1)= \sum_{i=1}^{2} (-1)^{i} \frac{(2)!}{i!(3-i)!} y_1^{3(3-i)+i} = -y_1^7+y_1^5.
\end{equation*}

As before, let $B_1$ be the branch locus of $\phi_1: Y_1 \to \PP^1$. For $P \in B_1$, let $d_P$ be the unique break in the lower-numbering ramification filtration at $P$. Minimality of the cover (see Definition~\ref{defnminimalcyclic}) implies $d_P | p-1$. We may assume $f$ has a pole of order $d_P$ at every $P \in B_1$ and no poles elsewhere. By Theorem~\ref{theoremcaljumpirred}, the pole order of $h$ at $P$ can be assumed to be less than $pd_P$. Without loss of generality, we assume $\infty \in B_1$.

Similarly, let $B_2$ be the branch locus of the cover $\phi_2: Y_2 \to \PP^1$. For $P \in B_2 \setminus B_1$, let $e_P$ be the ramification jump of the cover $Y_2 \to Y_1$ at $P$. By minimality, we have $e_P | p-1$. 

In Theorem~\ref{thmp=3} we prove that the $a$-number of $Y_2$ depends only on the ramification jumps $d_P$ and $e_P$. We do this by showing that the lower bound provided by Theorem~\ref{thmrkCgeqK} is sharp. We compute the basis of differentials $W_P$ given in Equation~\eqref{eqdefWinf} and Equation~\eqref{eqdefWP}. This results in six cases. We distinguish between $P=\infty$, $P \in B_1 \setminus \{\infty\}$ and $P \in B_2 \setminus B_1$. Moreover, we distinguish different cases depending on whether the ramification jump ($d_P$ or $e_P$) equals $1$ or $2$. These cases are treated in the following subsections. All of this information is combined to prove Theorem~\ref{thmp=3}.

In the case $p=3$, Lemma~\ref{lemcomnon-zero} guarantees that the coefficients $c_\om$ are non-zero. Note that by construction we have $0 \leq \be_P(v) \leq 2$. In the case $\be_P(v)=0$, the condition $\be_P(v) \leq a_1$ is satisfied. Otherwise, the condition $\be_P(v) \geq p-2=1$ is satisfied, so Lemma~\ref{lemcomnon-zero} always applies. This is used throughout this section.

\subsection{A pole of \texorpdfstring{$f$}{} of order \texorpdfstring{$1$}{} at infinity} \label{subsec:P=infd=1}

This section concerns the action of $\C$ on $W_\infty$, assuming $f$ has a pole at $\infty$ of order $d_\infty=1$. The definition of $W_\infty$, from Equation~\eqref{eqdefWinf}, becomes
\begin{align*}
W_\infty &= \begin{aligned}
     \begin{cases} y_2^{a_2} y_1^{a_1} x^v dx \; | \; 0 \leq a_1,a_2 <3, \\  
       0 \leq 9 v \leq  3(2-a_1) + 7(2-a_2)-10 \end{cases}  \Bigg\} \end{aligned} \\
       &= \{dx, xdx, y_1dx, y_1^2 dx, y_2 dx, y_2y_1dx\}.
\end{align*}

We study the restricted Cartier operator $\C_{W_\infty}: \Span_k W_\infty \to \cohom^0(Y_2, \Om_{Y_2}^1)$, using machinery from Section \ref{Sec:Keyterms}. For each element of $W_\infty$, we check whether it is in $H_\infty$. For the differentials $\om$ that lie in $H_\infty$, we compute the key term $\ka(\C(\om))$ and the coefficient $c_\om$. This results in Table~\ref{tablep=3P=infd=1}. 
\begin{table}[h] \label{tablep=3P=infd=1}
\centering
\begin{tabular}{|l|l|l|l|l|l|l|l|l|}
$a_2$ & $a_1$ & $v$ & $\omega$    & $\al_\infty(v)$ & $\be_\infty(v)$ & $\ka(\C(\om))$ & $c_{\om}$                                 \\ \hline
$0$ & $0$ & $0$ & $dx$       & & & &   \\ \hline
$0$ & $0$ & $1$ & $xdx$      & & & &   \\ \hline
$0$ & $1$ & $0$ & $y_1dx$    & & & &   \\ \hline
$0$ & $2$ & $0$ & $y_1^2dx$  & $0$ & $2$ & $dx$ &            $\binom{2}{2}=1$  \\\hline
$1$ & $0$ & $0$ & $y_2 dx$   & $0$ & $2$ & $y_1dx$& $-\frac{2!}{2!1!}\binom{2}{2}=-1$ \\ \hline
$1$ & $1$ & $0$ & $y_2y_1dx$  & $0$ & $2$ & $y_1^2 dx$ &              $\frac{2!}{1!2!}\binom{2}{2}=1$  \\ \hline
\end{tabular}
\caption{The computation of key terms in the case $P=\infty$ and $d_\infty=1$.}
\end{table}
There are three different key terms, which contribute $3$ to the rank of $\C$ via Theorem~\ref{thmrkCgeqK}. In order to show later that the remainig elements of $W_\infty$ do not contribute to the rank, we apply the Cartier operator to the three elements of $W_\infty \setminus H_\infty$.
\begin{align*}
    \C(dx) &= 0 \\
    \C(xdx) &= 0 \\
    \C(y_1dx) &= \C((y_1^p-f(x))dx) = y_1 \C(dx) - \C(f(x)dx) \\
    &= -\sum_{Q \in B} \C (f_Q(x_Q)dx) \in \Span_k \{ x_Qdx \; | \; Q \in B_1 \setminus \{\infty\} \}.
\end{align*}

\subsection{A pole of \texorpdfstring{$f$}{} of order \texorpdfstring{$1$}{} away from infinity} \label{subsec:Pnotinfd=1}

We now consider the case $d_P=1$, where $P \neq \infty$. We compute the basis
$$\begin{aligned}
    W_P=\begin{cases} y_2^{a_2} y_1^{a_1} x_P^v dx \; | \; 0 \leq a_1,a_2 <3, \\  
       0 < 9 v \leq  3(2-a_1) + 7(2-a_2)+8 \end{cases}  \Bigg\}.
\end{aligned}$$

Table~\ref{tablep=3Pnotinfd=1} shows the $14$ differentials and their key terms, if they exist. The coefficients $c_\omega$ are omitted, as Lemma~\ref{lemcomnon-zero} provides that they are non-zero.

\begin{table}[h] \label{tablep=3Pnotinfd=1}
\centering
\begin{tabular}{|l|l|l|l|l|l|l|l|l|}
$a_2$ & $a_1$ & $v$ & $\omega$   & $\al_P(v)$ & $\be_P(v)$ & $\ka(\C(\om))$                              \\ \hline
$0$ & $0$ & $1$ & $x_Pdx$          & $0$ & $0$ & $x_Pdx$         \\ \hline
$0$ & $0$ & $2$ & $x_P^2dx$        & $0$ & $2$ &                 \\ \hline
$0$ & $0$ & $3$ & $x_P^3dx$        & $1$ & $1$ &                 \\ \hline
$0$ & $1$ & $1$ & $y_1x_Pdx$       & $0$ & $0$ & $y_1x_Pdx$      \\ \hline
$0$ & $1$ & $2$ & $y_1x_P^2dx$     & $0$ & $2$ &                 \\ \hline
$0$ & $2$ & $1$ & $y_1^2x_Pdx$     & $0$ & $0$ & $y_1^2x_Pdx$    \\ \hline
$0$ & $2$ & $2$ & $y_1^2x_P^2dx$   & $0$ & $2$ & $x_P^2dx$       \\ \hline
$1$ & $0$ & $1$ & $y_2x_Pdx$       & $0$ & $0$ & $y_2x_Pdx$      \\ \hline
$1$ & $0$ & $2$ & $y_2x_P^2dx$     & $0$ & $2$ & $y_1x_P^2dx$    \\ \hline
$1$ & $1$ & $1$ & $y_2y_1x_Pdx$    & $0$ & $0$ & $y_2y_1x_Pdx$   \\ \hline
$1$ & $1$ & $2$ & $y_2y_1x_P^2dx$ & $0$ & $2$ & $y_1^2x_P^2dx$  \\ \hline
$1$ & $2$ & $1$ & $y_2y_1^2x_Pdx$   & $0$ & $0$ & $y_2y_1^2x_Pdx$ \\ \hline
$2$ & $0$ & $1$ & $y_2^2x_Pdx$     & $0$ & $0$ & $y_2^2x_Pdx$    \\ \hline
$2$ & $1$ & $1$ & $y_2^2y_1x_Pdx$  & $0$ & $0$ & $y_2^2y_1x_Pdx$ \\ \hline
\end{tabular}
\caption{The computation of key terms in the case $P \in B_1 \setminus \{\infty\}$ and $d_P=1$.}
\end{table}

The table shows that all the $11$ key terms are different. The basis elements without a key term require more attention.
\begin{align*}
    \C(x_P^2 dx)     &= 0 \\
    \C(x_P^3 dx)     &= 0 \\ 
    \C(y_1 x_P^2 dx) &= y_1 \C(x_P^2 dx) - \C(x_P^2 f(x) dx) \\
    & \in \begin{cases} 
    \Span_k \{x_P dx, x_Q dx \; | \; Q \in B_1 \setminus \{P,\infty\} \} &\hbox{if $d_\infty = 1$} \\
    \Span_k \{dx, x_P dx, x_Q dx \; | \; Q \in B_1 \setminus \{P, \infty\}\} &\hbox{if $d_\infty = 2$}.
    \end{cases}
\end{align*}

\subsection{A pole of \texorpdfstring{$f$}{} of order \texorpdfstring{$2$}{} at infinity} \label{subsec:P=infd=2}

Assume $f$ has a pole of order $2$ at $P=\infty$. We then compute the basis
$$\begin{aligned}
    W_\infty = \begin{cases} y_2^{a_2} y_1^{a_1} x^v dx \; | \; 0 \leq a_1,a_2 <3, \\  
       0 \leq 9 v \leq  6(2-a_1) + 14(2-a_2)-10 \end{cases}  \Bigg\}.
\end{aligned}$$


Table~\ref{tablep=3P=infd=2} shows these $16$ basis differentials and their key terms, if they exist. 

\begin{table}[h] \label{tablep=3P=infd=2}
\centering
\begin{tabular}{|l|l|l|l|l|l|l|l|l|}
$a_2$ & $a_1$ & $v$ & $\omega$   & $\al_\infty(v)$ & $\be_\infty(v)$ & $\ka(\C(\om))$                              \\ \hline
$0$ & $0$ & $0$ & $dx$         & $0$ & $1$ &           \\ \hline
$0$ & $0$ & $1$ & $xdx$        & $0$ & $2$ &            \\ \hline
$0$ & $0$ & $2$ & $x^2dx$      & $1$ & $0$ & $dx$            \\ \hline
$0$ & $0$ & $3$ & $x^3dx$      & $1$ & $1$ &             \\ \hline
$0$ & $1$ & $0$ & $y_1dx$      & $0$ & $1$ & $dx$            \\ \hline
$0$ & $1$ & $1$ & $y_1xdx$     & $0$ & $2$ &            \\ \hline
$0$ & $1$ & $2$ & $y_1x^2dx$   & $1$ & $0$ & $y_1dx$         \\ \hline
$0$ & $2$ & $0$ & $y_1^2dx$    & $0$ & $1$ & $y_1dx$         \\ \hline
$0$ & $2$ & $1$ & $y_1^2xdx$   & $0$ & $2$ & $xdx$           \\ \hline
$0$ & $2$ & $2$ & $y_1^2x^2dx$ & $1$ & $0$ & $y_1^2dx$       \\ \hline
$1$ & $0$ & $0$ & $y_2dx$      & $0$ & $1$ & $y_1^2dx$       \\ \hline
$1$ & $0$ & $1$ & $y_2xdx$     & $0$ & $2$ & $y_1xdx$        \\ \hline
$1$ & $1$ & $0$ & $y_2y_1dx$   & $0$ & $1$ & $y_2dx$         \\ \hline
$1$ & $1$ & $1$ & $y_2y_1xdx$  & $0$ & $2$ & $y_1^2xdx$      \\ \hline
$1$ & $2$ & $0$ & $y_2y_1^2dx$ & $0$ & $1$ & $y_2y_1dx$      \\ \hline
$2$ & $0$ & $0$ & $y_2^2dx$    & $0$ & $1$ & $y_2y_1^2dx$    \\ \hline
\end{tabular}
\caption{The computation of key terms in the case $P=\infty$ and $d_\infty=2$.}
\end{table}

Note that $12$ differentials have a key term, but there are only $9$ distinct key terms. The differentials outside $H_\infty$ and the pairs of differentials with the same key term require extra attention.
\begin{align*}
    \C(dx) &= 0 \\
    \C(xdx) &= 0 \\
    \C(x^2 dx) &= dx \\
    \C(x^3 dx) &= 0 \\
    \C(y_1dx) &= -\C(f(x)dx) \in \Span_k \{dx, x_Q dx \; | \; Q \in B_1 \setminus \{ \infty\} \} \\
    \C(y_1xdx) &= - \C(xf(x)dx) \in \Span_k \{dx, x_Q dx \; | \; Q \in B_1 \setminus \{ \infty\} \} \\
    \C(y_1 x^2 dx) &= y_1 dx - \C(x^2f(x) dx) \in \Span_k \{dx, y_1 dx, x_Q dx \; | \; Q \in B_1 \setminus \{ \infty \} \} \\
    \C(y_1^2 dx) &= y_1\C(f(x)dx) + \C(f(x)^2dx) \in \Span_k \{dx, y_1 dx, x_Qdx, y_1x_Q dx \; | \; Q \in B_1 \setminus \{ \infty \} \} \\ 
    &\hspace{5mm}  \oplus \Span_k \{x_Q^2dx \: | \; Q \in B_1 \setminus \{ \infty\}, d_Q=2\} \\
    \C(y_1^2x^2dx) &=y_1^2 \C(x^2 dx) + y_1 \C(x^2 f(x)dx) + \C(x^2f(x)^2dx) \\
    &\in \Span_k \{dx, xdx, y_1dx, y_1^2dx, x_Q dx, y_1x_Q dx, \; | \; Q \in B_1 \setminus \{ \infty \} \} 
    \\ &\hspace{5mm} \oplus \Span_k\{x_Q^2dx \; | \; Q \in B_1 \setminus \{ \infty\} , d_Q=2 \} \\
    \C(y_2 dx) &= \C(y_1^7 dx) - \C(y_1^5 dx) - \C(h(x)dx) \\ 
    &= y_1^2 \C(y_1dx) - y_1 \C( y_1^2 dx) - \C(h(x)dx) \\ 
    &\in \Span_k \{dx, xdx, y_1 dx, y_1^2dx, x_Qdx, y_1x_Q dx, y_1^2x_Q dx \; | \; Q \in B_1 \setminus \{ \infty \} \} \\ 
    &\hspace{5mm} \oplus \Span_k \{x_Q^2dx, y_1 x_Q^2 dx \; | \; Q \in B_1 \setminus \{ \infty \}, d_Q = 2\} \\
    &\hspace{5mm} \oplus \Span_k \{x_R dx \; | \; R \in B_2 \setminus B_1\}.
\end{align*}

\subsection{A pole of \texorpdfstring{$f$}{} of order \texorpdfstring{$2$}{} away from infinity} \label{subsec:Pnotinfd=2}

We now consider the case when $f$ has a pole at $P \neq \infty$ of order $d_P=2$. The basis is
$$\begin{aligned}
    W_P=\begin{cases} y_2^{a_2} y_1^{a_1} x_P^v dx \; | \; 0 \leq a_1,a_2 <3, \\  
       0 < 9 v \leq  6(2-a_1) + 14(2-a_2)+8 \end{cases}  \Bigg\}.
\end{aligned}$$

The $24$ differentials and their key terms, if they exist, are compiled in Table~\ref{tablep=3Pnotinfd=2}.

\begin{table}[h] \label{tablep=3Pnotinfd=2}
\centering
\begin{tabular}{|l|l|l|l|l|l|l|l|l|}
$a_2$ & $a_1$ & $v$ & $\omega$   & $\al_P(v)$ & $\be_P(v)$ & $\ka(\C(\om))$                              \\ \hline
$0$ & $0$ & $1$ & $x_Pdx$          & $0$ & $0$ & $x_Pdx$         \\ \hline
$0$ & $0$ & $2$ & $x_P^2dx$        & $0$ & $1$ &                 \\ \hline
$0$ & $0$ & $3$ & $x_P^3dx$        & $0$ & $2$ &                 \\ \hline
$0$ & $0$ & $4$ & $x_P^4dx$        & $1$ & $0$ & $x_P^2dx$       \\ \hline
$0$ & $0$ & $5$ & $x_P^5dx$        & $1$ & $1$ &                 \\ \hline
$0$ & $1$ & $1$ & $y_1x_Pdx$       & $0$ & $0$ & $y_1x_Pdx$      \\ \hline
$0$ & $1$ & $2$ & $y_1x_P^2dx$     & $0$ & $1$ & $x_P^2dx$       \\ \hline
$0$ & $1$ & $3$ & $y_1x_P^3dx$     & $0$ & $2$ &                 \\ \hline
$0$ & $1$ & $4$ & $y_1x_P^4dx$     & $1$ & $0$ & $y_1x_P^2dx$    \\ \hline
$0$ & $2$ & $1$ & $y_1^2x_Pdx$     & $0$ & $0$ & $y_1^2x_Pdx$    \\ \hline
$0$ & $2$ & $2$ & $y_1^2x_P^2dx$   & $0$ & $1$ & $y_1x_P^2dx$    \\ \hline
$0$ & $2$ & $3$ & $y_1^2x_P^3dx$   & $0$ & $2$ & $x_P^3dx$       \\ \hline
$0$ & $2$ & $4$ & $y_1^2x_P^4dx$   & $1$ & $0$ & $y_1^2x_P^2dx$  \\ \hline
$1$ & $0$ & $1$ & $y_2x_Pdx$       & $0$ & $0$ & $y_2x_Pdx$      \\ \hline
$1$ & $0$ & $2$ & $y_2x_P^2dx$     & $0$ & $1$ & $y_1^2x_P^2dx$  \\ \hline
$1$ & $0$ & $3$ & $y_2x_P^3dx$     & $0$ & $2$ & $y_1x_P^3dx$    \\ \hline
$1$ & $1$ & $1$ & $y_2y_1x_Pdx$    & $0$ & $0$ & $y_2y_1x_Pdx$   \\ \hline
$1$ & $1$ & $2$ & $y_2y_1x_P^2dx$ & $0$ & $1$ & $y_2x_P^2dx$    \\ \hline
$1$ & $1$ & $3$ & $y_2y_1x_P^3dx$ & $0$ & $2$ & $y_1^2x_P^3dx$  \\ \hline
$1$ & $2$ & $1$ & $y_2y_1^2x_Pdx$   & $0$ & $0$ & $y_2y_1^2x_Pdx$ \\ \hline
$1$ & $2$ & $2$ & $y_2y_1^2x_P^2dx$ & $0$ & $1$ & $y_2y_1x_P^2dx$ \\ \hline
$2$ & $0$ & $1$ & $y_2^2x_Pdx$     & $0$ & $0$ & $y_2^2x_Pdx$    \\ \hline
$2$ & $0$ & $2$ & $y_2^2x_P^2dx$   & $0$ & $1$ & $y_2y_1^2x_P^2dx$ \\ \hline
$2$ & $1$ & $1$ & $y_2^2y_1x_Pdx$  & $0$ & $0$ & $y_2^2y_1x_Pdx$ \\ \hline
\end{tabular}
\caption{The computation of key terms in the case $P \in B_1 \setminus \{\infty\}$ and $d_P=2$.}
\end{table}

The table displays that $20$ differentials have a key term, but there are only $17$ distinct key terms. The differentials without a key term and the pairs of differentials with the same key term require extra attention.
\begin{align*}
    \C(x_P^2 dx)      &= 0 \\
    \C(x_P^3 dx)      &= 0 \\ 
    \C(x_P^4 dx)      &= x_P^2dx \\
    \C(x_P^5 dx)      &= 0 \\
    \C(y_1x_P^2dx)    &= y_1 \C(x_P^2 dx) - \C(x_P^2f(x) dx) \\ 
    & \hspace{-16mm} \in \begin{cases} 
    \Span_k \{ x_P dx, x_P^2 dx, x_Q dx \; | \; Q \in B_1 \setminus \{P, \infty\} \} &\hbox{if $d_\infty=1$} \\
    \Span_k \{dx,  x_P dx, x_P^2 dx, x_Q dx \; | \; Q \in B_1 \setminus \{P, \infty\} \} &\hbox{if $d_\infty=2$}
    \end{cases} \\
    \C(y_1x_P^3 dx) &= y_1 \C(x_P^3 dx) - \C(x_P^3f(x) dx) \\
    & \hspace{-16mm} \in \begin{cases} 
    \Span_k \{ x_P dx, x_P^2dx, x_Q dx \; | \; \in B_1 \setminus \{P, \infty \} \} &\hbox{if $d_\infty=1$} \\
    \Span_k \{dx,  x_P dx, x_P^2 dx, x_Q dx \; | \; Q \in B_1 \setminus \{P,\infty\}\} &\hbox{if $d_\infty=2$}
    \end{cases} \\
    \C(y_1x_P^4 dx) &= y_1 \C(x_P^4 dx)- \C(x_P^4 f(x) dx) \\
    &\hspace{-16mm} \in \begin{cases} 
    \Span_k \{ y_1x_P^2dx, x_P dx, x_P^2 dx, x_Q dx \; | \; Q \in B_1 \setminus \{P, \infty \} \} &\hbox{if $d_\infty=1$} \\
    \Span_k \{ y_1x_P^2dx, dx, x_P dx, x_P^2 dx, x_Q dx \; | \; Q \in B_1 \setminus \{P,\infty\} \} &\hbox{if $d_\infty=2$}
    \end{cases} \end{align*} \begin{align*}
    \C(y_1^2x_P^2dx) &= y_1^2\C(x_P^2 dx) + y_1 \C(x_P^2 f(x) dx) + \C(x_P^2 f(x)^2 dx) \\
    &\hspace{-16mm}\in \begin{cases} 
    \Span_k \{ y_1x_P^2dx, dx, x_P dx, x_P^2 dx, y_1x_Pdx, x_Q dx, y_1x_Qdx \; | \; Q \in B_1 \setminus \{P, \infty \} \}  \\ 
    \oplus \Span_k \{ x_Q^2dx \; Q \in B_1 \setminus \{P, \infty \} , d_Q=2 \} \hspace{20mm} \hbox{if $d_\infty=1$} \\
    \Span_k \{ y_1x_P^2dx, dx, y_1dx, x_P dx, x_P^2 dx, y_1x_Pdx, x_Q dx, y_1x_Qdx \; | \; Q \in B_1 \setminus \{P, \infty \}\} \\ 
    \oplus \Span_k \{ x_Q^2dx \; Q \in B_1 \setminus \{P, \infty\} , d_Q=2 \}  \hspace{20mm} \hbox{if $d_\infty=2$}
    \end{cases} \\
    \C(y_1^2x_P^4dx) &= y_1^2 \C(x_P^4dx) + y_1 \C(x_P^4f(x)dx) + \C(x_P^4 f(x)^2)dx \\
    &\hspace{-16mm} \in \begin{cases} 
    \Span_k \{ y_1^2x_P^2dx, dx, x_Pdx, x_P^2dx, x_P^3dx, y_1x_P^2dx, x_Qdx, y_1x_Qdx \; | \; Q \in B_1 \setminus \{P, \infty \} \} \\
    \oplus \Span_k\{x_Q^2 dx \; | \; Q\in B_1 \setminus \{P, \infty \},  d_Q = 2 \} \hspace{20mm} \hbox{if $d_\infty=1$} \\
    \Span_k \{ y_1^2x_P^2dx, dx, y_1dx, x_Pdx, x_P^2dx, x_P^3 dx, y_1x_P^2, x_Qdx, y_1x_Qdx \; | \; Q \in B_1 \setminus \{P, \infty \} \} \\ 
    \oplus \Span_k \{ x_Q^2 dx \; | \; Q\in B_1 \setminus \{P, \infty \}, d_Q =2 \} \hspace{20mm} \hbox{if $d_\infty=2$} 
    \end{cases} \end{align*} \begin{align*}
    \C(y_2x_P^2dx) &= y_1^2 \C( y_1 x_P^2 dx) - y_1 \C(y_1^2x_P^2dx) - \C(x_P^2 h(x) dx) \\ 
    & \hspace{-16mm} \in \begin{cases} 
    \Span_k \{ y_1^2x_P^2dx, dx, y_1dx, x_Pdx, x_p^2dx, x_P^3dx, y_1x_Pdx, y_1x_P^2dx, y_1^2x_Pdx \} \\
    \oplus \Span_k \{x_Qdx, y_1x_Q dx, y_1^2x_Qdx \; | \; Q \in B_1 \setminus \{P, \infty \} \} \oplus \Span_k \{x_Rdx \: | \; R \in B_2 \setminus B_1\} \\ 
    \oplus \Span_k \{x_Q^2dx, y_1x_Q^2 dx \; | \; Q\in B_1 \setminus \{P, \infty\},  d_Q = 2 \} \hspace{20mm} \hbox{if $d_\infty=1$} \\
    \Span_k \{ y_1^2x_P^2dx, dx, xdx, y_1dx, y_1^2dx, x_Pdx, x_p^2dx, x_P^3dx, y_1x_Pdx, y_1x_P^2dx, y_1^2x_Pdx \} \\
    \oplus \Span_k \{x_Qdx, y_1x_Q dx, y_1^2x_Qdx \; | \; Q \in B_1 \setminus \{P, \infty\} \} \oplus \Span_k \{x_Rdx \: | \; R \in B_2 \setminus B_1\} \\
    \oplus \Span_k \{x_Q^2dx, y_1x_Q^2 dx \; | \; Q\in B_1 \setminus \{P,\infty\},  d_Q = 2 \}  \hspace{20mm} \hbox{if $d_\infty=2$.} 
    \end{cases}
\end{align*}

\subsection{A pole of \texorpdfstring{$h$}{} of order \texorpdfstring{$1$}{}} \label{subsec:eP=1}

Assume $P \in B_2 \setminus B_1$ and $e_P=1$, meaning $h(x)$ has a pole of order $1$ at $P$. The basis $W_P$ is given by
$$\begin{aligned}
    W_P=\begin{cases} y_2^{a_2} y_1^{a_1} x_P^v dx \; | \; 0 \leq a_1,a_2 <3, \\  
       0 < 3v \leq  4-a_2 \end{cases}  \Bigg\}.
\end{aligned}$$

The value of $a_1$ plays no role in whether $\om$ is regular and in the key term $\ka(\C(\om))$ as defined in Definition~\ref{defktpoleh}. In Table~\ref{tablep=3Pnotinfd=1e=3}, the three possibilities $0 \leq a_1 \leq 2$ are combined into one line. 

\begin{table}[h] \label{tablep=3Pnotinfd=1e=3}
\centering
\begin{tabular}{|l|l|l|l|l|l|l|l|l|}
$a_2$ & $v$ & $\omega$   & $\al_P(v)$ & $\be_P(v)$ & $\ka(\C(\om))$                              \\ \hline
$0$ & $1$ & $y_1^{a_1}x_Pdx$         & $0$ & $0$ & $y_1^{a_1}x_P dx$          \\ \hline
$1$ & $1$ & $y_2y_1^{a_1}x_P dx$      & $0$ & $0$ & $y_2y_1^{a_1}x_P dx$       \\ \hline
\end{tabular}
\caption{The computation of key terms in the case $P \in B_2 \setminus B_1$, and $e_P=1$.}
\end{table}

Observe that these six differentials all have a unique key term.

\subsection{A pole of \texorpdfstring{$h$}{} of order \texorpdfstring{$2$}{}} \label{subsec:eP=2}

Assume $P \in B_2 \setminus B_1$ and $e_P=2$, meaning $h(x)$ has a pole of order $2$ at $P$. The basis is
$$\begin{aligned}
    W_P=\begin{cases} y_2^{a_2} y_1^{a_1} x_P^v dx \; | \; 0 \leq a_1,a_2 <3, \\  
       0 < 3 v \leq  6-2a_2 \end{cases}  \Bigg\}.
\end{aligned}$$

The differentials in $W_P$ with their key terms, if they exist, are presented in Table~\ref{tablep=3Pnotinfd=2e=3}. The basis $W_P$ consists of nine differentials and six of them have a key term. All of the key terms are unique. Again, $a_1$ can take any value in $\{0,1,2\}$.

\begin{table}[h] \label{tablep=3Pnotinfd=2e=3}
\centering
\begin{tabular}{|l|l|l|l|l|l|l|l|l|}
$a_2$ & $v$ & $\omega$   & $\al_P(v)$ & $\be_P(v)$ & $\ka(\C(\om))$                              \\ \hline
$0$ & $1$ & $y_1^{a_1}x_Pdx$         & $0$ & $0$ & $y_1^{a_1}x_P dx$          \\ \hline
$0$ & $2$ & $y_1^{a_1}x_P^2 dx$        & $0$ & $1$ &          \\ \hline
$1$ & $1$ & $y_2y_1^{a_1}x_P dx$      & $0$ & $0$ & $y_2y_1^{a_1}x_P dx$            \\ \hline
\end{tabular}
\caption{The computation of key terms in the case $P \in B_2 \setminus B_1$ and $e_P = 2$.}
\end{table}

The differentials without a key term require extra attention.
\begin{align*}
\C(x_P^2dx) &= 0 \\
\C(y_1x_P^2dx) &= y_1\C(x_P^2dx) - \C(f(x)x_P^2dx) \\
& \in \begin{cases}
    \Span_k\{x_P dx\} \oplus \Span_k \{x_Q dx \; | \; Q \in B_1 \setminus \{\infty\}\} &\hbox{if $d_\infty = 1$} \\
    \Span_k\{dx,x_P dx\} \oplus \Span_k \{x_Q dx \; | \; Q \in B_1 \setminus \{\infty\}\}  &\hbox{if $d_\infty = 2$} 
\end{cases} 
\end{align*}
\begin{align*}
\C(y_1^2x_P^2dx)&= y_1^2\C(x_P^2dx) + y_1 \C(f(x)x_P^2dx) +\C(f(x)^2x_P^2dx) \\
&\in \begin{cases}
\Span_k \{dx,x_Pdx,y_1x_Pdx\} \oplus \Span_k\{x_Qdx, y_1x_Qdx \; | \; Q \in B_1 \setminus \{\infty\} \}  \\
\oplus \Span_k\{x_Q^2dx \; | \; Q \in B_1 \setminus \{\infty\}, d_Q = 2\} \hspace{20mm} \hbox{if $d_\infty =1$} \\
\Span_k \{dx,y_1dx, x_Pdx, y_1x_Pdx\} \oplus \Span_k\{x_Qdx, y_1x_Qdx \; | \; Q \in B_1 \setminus \{\infty\} \}  \\
\oplus \Span_k\{x_Q^2dx \; | \; Q \in B_1 \setminus \{\infty\}, d_Q = 2\} \hspace{20mm} \hbox{if $d_\infty =2$.}
\end{cases}
\end{align*}

\subsection{The main theorem}

Let $Y_2 \to \PP^1$ be a $\Z/9\Z$-cover which is minimal in the sense of Definition~\ref{defnminimalcyclic}. We now prove a formula for the $a$-number of $Y_2$, implying that Conjecture~\ref{conjnumberminimal} holds under the assumptions $p=3$ and $n=2$.

\begin{theorem} \label{thmp=3}
    Suppose $Y_2 \xrightarrow{} \mathbb{P}^1$ has branching datum
\begin{equation*}
  \left[
\begin{array}{rrrrrrrr}
2&  \ldots  & 3 & \ldots & 0 & \ldots & 0 & \ldots \\ 
\undermat{n_1}{4 &  \ldots} & \undermat{n_2}{7 &  \ldots} &  \undermat{n_3}{2 &  \ldots}  &  \undermat{n_4}{{3} &  \ldots}\\ 
\end{array} 
\right]^{\intercal}
\end{equation*}
\vspace{3mm}

\noindent Then the $a$-number of $Y_2$ satisfies the formula
\begin{equation*}
    a_{Y_2}=3n_1+7n_2+0n_3+3n_4.
\end{equation*}
\end{theorem}
\begin{proof}
We apply Theorem~\ref{thmrkCgeqK} to obtain a lower bound for $\rk(\C)$ and hence an upper bound for $a_{Y_2}$. Combining Sections \ref{subsec:P=infd=1}, \ref{subsec:Pnotinfd=1}, \ref{subsec:P=infd=2} and \ref{subsec:Pnotinfd=2}, we find the number of distinct key terms is 
$$\#K = 11 n_1 + 17 n_2 + 6n_3+6n_4 -8.$$
The correcting factor $-8$ comes from the pole at infinity. Applying Theorem~\ref{thmrkCgeqK}, using $g_{Y_2}=\# W$, yields
\begin{align*}
a_{Y_2} &\leq g_{Y_2} - \#K \\ 
&= 3n_1+7n_2+0n_3+3n_4.
\end{align*}
The remainder of the proof revolves around showing that this upper bound is sharp. We do so by showing two things. First, we show that differentials which have no key term do not contribute to the rank of $\C$. Second, we show that differentials with the same key term together contribute $1$ to the rank of $\C$, such that the rank of $\C$ is exactly the number of key terms. We do this locally at the points in $B_2$, by treating the six cases corresponding to Sections~\ref{subsec:P=infd=1}, \ref{subsec:Pnotinfd=1}, \ref{subsec:P=infd=2}, \ref{subsec:Pnotinfd=2}, \ref{subsec:eP=1} and \ref{subsec:eP=2}.

We begin with the differentials in $W_\infty$ in the case $d_\infty=1$, as described in Section \ref{subsec:P=infd=1}. All the key terms occur uniquely, so we treat only the differentials that have no key term. As the differentials $dx$ and $xdx$ are killed by $\C$, they clearly don't contribute to the rank. Next, we consider $y_1dx$, which also does not have a key term. We have
$$\C(y_1 dx) \in \Span_k \{ x_Q dx \; | \; Q \neq \infty \}.$$
However, for $Q\neq \infty$, we have $\C(x_Qdx) = x_Q dx$ and $x_Q dx$ does have a key term. Hence $y_1dx$ also does not contribute to the rank of $\C$. Thus we have exhibited $3$ linearly independent differentials that don't contribute to $\rk(\C)$, meaning they contribute to $a_{Y_2}$.

Second, we consider the case $P\in B_1 \setminus \{\infty\}$ and $d_P=1$. Section \ref{subsec:Pnotinfd=1} shows that all the key terms are unique. The differentials $x_P^2dx$, $x_P^3dx$ and $y_1x_P^2dx$ have no key term. The differentials $x_P^2dx$ and $x_P^3dx$ are killed by $\C$ and therefore don't contribute to $\rk(\C)$. Furthermore, we have computed
$$ \C(y_1 x_P^2 dx)\in \begin{cases} 
    \Span_k \{x_P dx, x_Q dx \; | \; P \neq Q \neq \infty\} &\hbox{if $d_\infty = 1$} \\
    \Span_k \{dx, x_P dx, x_Q dx \; | \; P \neq Q \neq \infty\} &\hbox{if $d_\infty = 2$}
\end{cases}$$
In the case $d_\infty=1$, we have already shown that this does not contribute to $\rk(\C)$. In the case $d_\infty=2$, we have $\C(x^2dx)=dx$. In both cases $y_1 x_P^2 dx$ does not contribute to $\rk(\C)$. Again, the differentials in $W_P$ contribute exactly $3$ to $a_{Y_2}$.

Third, we consider the differentials in $W_\infty$ in the case $d_\infty=2$. As shown in Section~\ref{subsec:P=infd=2}, the differentials $dx$, $xdx$ and $x^3dx$ are killed by $\C$. $x^2dx$ and $y_1dx$ both have the key term $dx$. We've seen that $\C(x^2dx)=dx$ and $$\C(y_1dx) \in \Span_k \{ dx,x_Q dx \; | \; Q \in B_1 \setminus \{\infty\}\}.$$ We have already established that $$\Span_k \{x_Qdx \; | \; Q \in B_1 \setminus \{\infty\}\} \subset \im (\C),$$ so the differentials $x^2dx$ and $y_1dx$ together contribute $1$ to $\rk(\C)$. It then immediately follows that $y_1x dx$, which has no key term, does not contribute to $\rk(\C)$, since we have computed $$\C(y_1 xdx) \in \Span_k \{dx, x_Q dx \; | \; Q \in B_1 \setminus \{\infty\}\}.$$ Next, the two differentials $y_1x^2dx$ and $y_1^2dx$ both have the key term $y_1dx$. We've computed
\begin{align*}
    \C(y_1x^2 dx) & \in \Span_k \{dx, y_1dx, x_Qdx \; | \; Q \in B_1 \setminus\{ \infty\} \} \\ 
    \C(y_1^2 x dx) & \in \Span_k \{ dx, y_1dx, x_Qdx, y_1x_Qdx \; | \; Q \in B_1 \setminus \{ \infty\} \}  \\
    &\hspace{5mm} \oplus \Span_k \{x_Q^2 dx \; | \; Q \in B_1 \setminus \{ \infty\}, d_Q=2\}. 
\end{align*}
We show that they contribute $1$ to $\rk(\C)$ together, by showing that $y_1x_Qdx$ and $x_Q^2 dx$ (for $d_Q=2$) are already in the image of $\C$. We compute
\begin{align*}
    \C(y_1 x_Q dx)&= y_1 x_Q dx - \C(x_Q f(x) dx)  \\
    &\in \Span_k \{ dx, y_1x_Qdx, x_P dx \; | \; P \in B_1 \setminus \{ \infty\} \} \\
    \C(x_Q^4 dx) &= x_Q^2 dx.
\end{align*}
The second computation assumes $d_Q=2$, so that $x_Q^4dx$ is indeed regular. These computations together imply that the differentials $y_1x^2dx$ and $y_1^2dx$, which have the same key term, together contribute $1$ to $\rk(\C)$. Finally, we consider the differentials $y_1^2x^2dx$ and $y_2dx$, which both have the key term $y_1^2dx$. In Section \ref{subsec:P=infd=2}, we have established
\begin{align*}
    \C(y_1^2x^2dx) &\in \Span_k \{dx, xdx, y_1dx, y_1^2dx, x_Q dx, y_1x_Q dx, \; | \; Q \in B_1 \setminus \{ \infty \} \} . 
    \\ &\hspace{5mm} \oplus \Span_k\{x_Q^2dx \; | \; Q \in B_1 \setminus \{ \infty \}, d_Q=2 \} \\
    \C(y_2 dx) &\in \Span_k \{dx, xdx, y_1 dx, y_1^2dx, x_Qdx, y_1x_Q dx, y_1^2x_Q dx \; | \; Q \in B_1 \setminus \{ \infty \} \} \\ 
    &\hspace{5mm} \oplus \Span_k \{x_Q^2dx, y_1 x_Q^2 dx \; | \; Q \in B_1 \setminus \{ \infty \}, d_Q = 2\} \\
    &\hspace{5mm} \oplus \Span_k \{x_R dx \; | \; R \in B_2 \setminus B_1\}.
\end{align*}
It remains to be shown that $xdx$,  $y_1^2x_Q dx$, $y_1 x_Q^2 dx$ (in the case $d_Q=2$) and $x_R dx$ (with $R \in B_2 \setminus B_1$) already lie in the image of $\C$. We compute
\begin{align*}
    \C(y_1^2 x dx) &= y_1^2 \C(xdx) + y_1 \C(xf(x)dx) + \C(xf(x)^2 dx) \\ 
    &\in \Span_k \{ dx, xdx, y_1 dx, x_Q dx, y_1 x_Q dx \; | \; Q \in B_1 \setminus \{ \infty\} \}  \\ 
    & \hspace{5mm} \oplus \Span_k \{ x_Q^2 dx \; | \; Q \in B_1 \setminus \{\infty\}, d_Q =2 \} \end{align*} \begin{align*}
    \C(y_1^2 x_Q dx) &= y_1^2x_Q dx + y_1 \C(x_Q f(x) dx) + \C(x_Q f(x)^2 dx) \\
    & \in \Span_k \{y_1^2x_Qdx, dx, y_1dx, x_Pdx, y_1x_Pdx \; | \; P \in B_1 \setminus \{\infty\} \} \\
    &\hspace{5mm} \oplus \Span_k \{ x_P^2 dx \; | \; P \in B_1 \setminus \{\infty\}, d_P=2 \} \\
    \C(y_1 x_Q^4 dx) &= y_1x_Q^2 dx - \C(x_Q^4 f(x) dx) \\ 
    &\in \Span_k \{ y_1x_Q^2dx, dx, x_Q^2dx, x_P dx \; | \; P \in B_1 \setminus \{ \infty \} \} \\
    \C(x_R dx) &= x_R dx.
\end{align*}
These computations verify that indeed $xdx$, $y_1^2x_Qdx$, $y_1x_Q^2dx$ and $x_Rdx$ already lie in the image of $\C$, without the help of $y_1^2x^2dx$ and $y_2dx$. Thus the differentials $y_1^2x^2dx$ and $y_2dx$ together contribute $1$ to $\rk(\C)$.

As the fourth step, we treat the case where $P \in B_1 \setminus \{\infty\}$ and $d_P=2$. In Section~\ref{subsec:Pnotinfd=2}, we have seen that $x_P^2dx$, $x_P^3dx$ and $x_P^5 dx$ are killed by the Cartier operator. The differentials $x_P^4dx$ and $y_1x_P^2$ have the key term $x_P^2dx$, but the computations in Section~\ref{subsec:Pnotinfd=2}, together with computations earlier in this proof, imply that they together contribute $1$ to $\rk(\C)$. It then follows immediately that $y_1x_P^3$, which has no key term, does not contribute to $\rk(\C)$. Next we consider the differentials $y_1x_P^4dx$ and $y_1^2x_P^2dx$, which both have the key term $y_1x_P^2dx$. The computations in Section~\ref{subsec:Pnotinfd=2} show that they together contribute $1$ to $\rk(\C)$; they both extend the imagee of $\C$ only by $\Span_k \{y_1x_P^2dx\}$ and nothing else. Finally, we consider the differentials $y_1^2x_P^4dx$ and $y_2x_P^2dx$, which both have the key term $y_1^2x_P^2dx$. They introduce not only $y_1^2x_P^2dx$, but also $x_P^3dx$ and $y_1dx$. We show that these differentials already lie in the image of $\C$ independently. For $x_P^3dx$, we compute
\begin{align*}
    \C(y_1^2x_P^3) &= y_1^2 \C(x_P^3dx) + y_1 \C(x_P^3f(x)dx) + \C(x_P^3f(x)^2dx) \\ 
    & \in \begin{cases} 
    \Span_k \{x_P^3dx, dx, x_Pdx, x_P^2dx, y_1x_Pdx, y_1x_P^2dx  \} \\
    \oplus \Span_k \{x_Qdx, y_1x_Q dx \; | \; Q \in B_1 \setminus \{P, \infty\} \} \\
    \oplus \Span_k \{x_Q^2dx \; | \; Q \in B_1 \setminus \{ \infty \},  d_Q = 2 \} \hspace{30mm} \hbox{if $d_\infty=1$} \\
    \Span_k \{ x_P^3dx, dx, y_1dx, x_Pdx, x_P^2dx, y_1x_Pdx, y_1x_P^2dx  \} \\
    \oplus \Span_k \{x_Qdx, y_1x_Q dx, y_1^2x_Qdx \; | \; Q \in B_1 \setminus \{P, \infty\} \} \\
    \oplus \Span_k \{x_Q^2dx, y_1x_Q^2 dx \; | \; Q\in B_1 \setminus \{\infty\},  d_Q = 2 \}  \hspace{17mm} \hbox{if $d_\infty=2$} 
    \end{cases}
\end{align*}
Hence $x_P^3dx$ was already in the image of $\C$. It remains to be shown that $y_1dx$ is already in the image of $\C$ in the case $d_\infty=1$. We compute 
\begin{align*}
    \C(y_1^2 dx) &= y_1^2\C(dx) + y_1 \C(f(x)dx) + \C(f(x)^2dx) \\
    &\in \span \{dx, x_Qdx, y_1x_Q dx \: | \; Q \in B_1 \setminus \{ \infty \} \} \oplus \Span_k \{ x_Q^2 dx \; | \; Q \in B_1 \setminus \{ \infty\}, d_Q=2 \} \\
    \C(y_2dx) &= y_1^2 \C(dx) - y_1 \C(y_1^2 dx) - \C(h(x)dx) \\
    & \in \Span_k \{ dx, y_1dx, x_Q dx, y_1x_Qdx, y_1^2x_Qdx \; | \; Q \in B_1 \setminus \{\infty\} \} \\
    &\oplus \Span_k \{ x_Q^2dx, y_1x_Q^2dx \; | \; Q \in B_1 \setminus \{ \infty \}, d_Q=2 \}.
\end{align*}
This shows that $y_1dx$ was already in the image of $\C$. It follows that the differentials $y_1^2x_P^4dx$ and $y_2x_P^2dx$ together contribute $1$ to $\rk(\C)$, which concludes the fourth step.

As the fifth step, we treat the case $P \in B_2 \setminus B_1$ and $e_P=1$, meaning $h$ has a pole of order $1$. In that case, Section~\ref{subsec:eP=1} shows that the differentials in $W_P$ all have different key terms. Hence these differentials do not contribute to the $a$-number.

As the sixth and last step, we treat the case $P \in B_2 \setminus B_1$ and $e_P=2$, meaning $h$ has a pole of order $2$. Section~\ref{subsec:eP=2} shows that $6$ out of the $9$ differentials in $W_P$ have a key term, and these are all unique. This implies that the rank of the Cartier operator restricted to $\Span_k W_P$ is at least $6$. We now show that it's exactly $6$, by analyzing the basis differentials that do not have a key term. Since $\C(x_P^2 dx)=0$, the differential $x_P^2dx$ does not contribute to the rank of $\C$. Moving on to $y_1x_P^2dx$, recall
\begin{align*}
    \C(yx_P^2dx) &= y\C(x_P^2dx) - \C(f(x)x_P^2dx) \\
& \in \begin{cases}
    \Span_k\{x_P dx\} \oplus \Span_k \{x_Q dx \; | \; Q \in B_1 \setminus \{\infty\}\} &\hbox{if $d_\infty = 1$} \\
    \Span_k\{dx,x_P dx\} \oplus \Span_k \{x_Q dx \; | \; Q \in B_1 \setminus \{\infty\}\}  &\hbox{if $d_\infty = 2$} 
\end{cases}
\end{align*}
Since $\C(x_P dx) = x_Pdx$, these differentials are already in the image of $\C$. Hence $y_1x_P^2dx$ does not contribute to the rank of $\C$. Finally, we treat $y_1^2x_P^2dx$:
\begin{align*}
\C(y_1^2x_P^2dx)&= y_1^2\C(x_P^2dx) + y_1 \C(f(x)x_P^2dx) +\C(f(x)^2x_P^2dx) \\
&\in \begin{cases}
\Span_k \{dx,x_Pdx,y_1x_Pdx\} \oplus \Span_k\{x_Qdx, y_1x_Qdx \; | \; Q \in B_1 \setminus \{\infty\} \}  \\
\oplus \Span_k\{x_Q^2dx \; | \; Q \in B_1 \setminus \{\infty\}, d_Q = 2\} \hspace{20mm} \hbox{if $d_\infty =1$} \\
\Span_k \{dx,y_1 dx, x_Pdx, y_1 x_Pdx\} \oplus \Span_k\{x_Qdx, y_1 x_Qdx \; | \; Q \in B_1 \setminus \{\infty\} \}  \\
\oplus \Span_k\{x_Q^2dx \; | \; Q \in B_1 \setminus \{\infty\}, d_Q = 2\} \hspace{20mm} \hbox{if $d_\infty =2$.}
\end{cases}
\end{align*}
To show that this does not contribute to the rank of $\C$, we only need to show that $y_1 x_P dx$ was already in the image of $\C$. This is clear since
$$\C(y_1 x_P dx) = y_1x_Pdx.$$
Thus $P$ contributes $3$ to the $a$-number of $Y_2$.

We have treated all six cases of $W_P$, depending on the ramification at $P$ and on whether $P=\infty$. We have shown that a pole of $f$ of order $2$ contributes $7$ to the $a$-number, while a pole of $f$ of order $1$ and a pole of $h$ of order $2$ both contribute $3$ to the $a$-number. A pole of $h$ of order $1$ contributes nothing to the $a$-number. In conclusion, we've proved the formula
$$a_{Y_2} = 3n_1 + 7n_2 + 3n_4.$$
\end{proof}

\begin{remark} \label{rmk:p=3imagenotkeyterm}
Seeing this proof, it would be tempting to think there is an inclusion $\im(\C) \subseteq \Span_k K$, where $K$ is the set of key terms. However, this is not true, even in the case $p=3$. To show this, take the differential $y_2^2dx$, assuming $d_\infty=2$. The differential $\C(y_2^2dx)$ has a term $\C(h(x)^2dx)$. The rational function $h(x)$ has a pole of order at most $5$ at infinity. This means for a suitable choice of $h(x)$, we have a term $\C(x^8dx)=x^2dx$, which does not occur as a key term.

Since $\im(\C)$ and $\Span_k K$ have the same dimension, by virtue of Theorem~\ref{thmp=3}, the opposite inclusion also fails to hold.
\end{remark}

\bibliographystyle{alpha}
\bibliography{mybibcopy}

\end{document}